\documentclass[11pt]{article}
\usepackage[T1]{fontenc}
\IfFileExists{lmodern.sty}{\usepackage{lmodern}}{}
\usepackage[margin=0.95in]{geometry}
\usepackage{amsmath,amssymb,amsthm}
\usepackage{booktabs,longtable,array}
\usepackage[hidelinks]{hyperref}
\newcommand{\Q}{\mathbb{Q}}
\newcommand{\C}{\mathbb{C}}
\newcommand{\Li}{\operatorname{Li}}
\newcommand{\Newt}{\operatorname{Newt}}
\newcommand{\inn}{\operatorname{in}}

\newcommand{\EE}{\mathbb{E}}
\newcommand{\ord}{\operatorname{ord}}
\newtheorem{theorem}{Theorem}[section]
\newtheorem{proposition}[theorem]{Proposition}
\newtheorem{lemma}[theorem]{Lemma}
\newtheorem{corollary}[theorem]{Corollary}

\theoremstyle{definition}

\title{A uniform degree bound for modular rank-two Nahm sums}
\author{Cetin Hakimoglu-Brown\\ \small\texttt{mathemails@proton.me}}
\date{September 28, 2026}

\begin{document}
\maketitle

\begin{abstract}
Let $F$ be a rank-two symmetrizable Nahm sum, with any symmetrizer $\operatorname{diag}(d_1,d_2)$, a positive-definite
rational matrix and rational shifts, and let $K$ be the field generated by its Nahm point. We prove that if $q^cF$ is modular, then
$[K:\Q]\le24$. Modularity forces the higher coefficients of the logarithmic radial asymptotic expansion to vanish; we
study the resulting polynomial equations through their Newton polygons and a pole-order analysis along the determinant
conic, and for $m=1$ show that their common components occur only on one explicit family, whose Nahm point is rational.
For equal row sums, Bloch-group torsion gives $[K:\Q]\le2$. For complementary Nahm coordinates, the bound is $2$
when $m\ne1$ and $4$ when $m=1$, using the first two asymptotic corrections in the latter case.
Several steps of the proofs rely on exact computation.
\end{abstract}

\medskip
\noindent\textit{2020 Mathematics Subject Classification.} Primary 11F03; Secondary 11P84, 14M25, 33D15.\\
\textit{Keywords.} Nahm sums, modular functions, radial asymptotics, Newton polygons, Bernoulli polynomials.

{\small\setcounter{tocdepth}{1}\tableofcontents}

\section{Introduction}\label{sec:setting}
Nahm sums are the $q$-series $\sum_{n\in\mathbb Z_{\ge0}^r}q^{\frac12n^{\mathsf T}An+b^{\mathsf T}n}/((q;q)_{n_1}\cdots(q;q)_{n_r})$
attached to a positive-definite matrix $A$ and a vector $b$. They arise as characters of rational conformal
field theories, and Nahm's conjecture \cite{Nah,Z} relates their modularity to torsion in the Bloch group of
the solutions of the associated Nahm equations; see \cite{VZ} for counterexamples to one direction and for
the asymptotic method used here. Calegari, Garoufalidis and Zagier \cite{CGZ} proved that modularity implies
that the associated element of the Bloch group is torsion. Mizuno \cite{Miz} extended the framework to
symmetrizable matrices, where the denominators become $(q^{d_i};q^{d_i})_{n_i}$, and proved the same
implication in that setting. A basic invariant of a modular Nahm sum is the number field
generated by its Nahm point, the solution of the Nahm equations in $(0,1)^r$. In all known modular examples
this field has small degree, but to our knowledge no uniform bound has been proved, even in rank two. In
this paper we prove such a bound for all rank-two symmetrizable Nahm sums with rational data. After a normalization
(Lemma~\ref{lem:norm}) the symmetrizer is $\operatorname{diag}(1,m)$ with $m\ge1$ rational; the case
$m=1$ contains all symmetric rank-two Nahm sums with rational data.

\paragraph{Setting.} Let $m\ge1$ be a rational number, $a,b,d\in\Q$ with $a,d>0$ and $\Delta:=ad-mb^2>0$, and $B,C\in\Q$. Put
\[
F(q)=\sum_{n,k\ge0}\frac{q^{\frac a2n^2+mbnk+\frac{md}2k^2+Bn+Ck}}{(q;q)_n(q^m;q^m)_k},
\]
with Nahm point $(x,y)\in(0,1)^2$ defined by $1-x=x^ay^b$, $1-y=x^{mb}y^d$, and $K=\Q(x,y)$.
In Mizuno's notation \cite{Miz} this is the Nahm sum of $(A,\mathbf b,D)$ with
$A=\bigl(\begin{smallmatrix}a&b\\ mb&d\end{smallmatrix}\bigr)$, $\mathbf b=(B,C)$, symmetrizer $D=\operatorname{diag}(1,m)$,
so that $AD$ is our positive-definite matrix; his Nahm equation is ours. For non-integral $m$, $F$ is a series in
fractional powers of $q$.

\begin{lemma}[Normalization of the symmetrizer]\label{lem:norm}
Let $d_1,d_2\ge1$ be integers, $m=d_2/d_1$, $A=\bigl(\begin{smallmatrix}a&b\\ mb&d\end{smallmatrix}\bigr)$ with
$A\operatorname{diag}(d_1,d_2)$ positive definite, $b_1,b_2,c\in\Q$, and let
\[
 f(q)=\sum_{n_1,n_2\ge0}\frac{q^{\frac12n^{\mathsf T}A\operatorname{diag}(d_1,d_2)n+b_1n_1+b_2n_2+c}}{(q^{d_1};q^{d_1})_{n_1}(q^{d_2};q^{d_2})_{n_2}}
\]
be the rank-two Nahm sum of \cite{Miz}. Then $f(q)=Q^{c/d_1}F(Q)$ with $Q=q^{d_1}$, where $F$ is the sum above with
this $m$ and $(B,C)=(b_1,b_2)/d_1$. The matrix $A$, and hence the Nahm point and the field $K$, are unchanged, and
$f$ is modular if and only if $Q^{c/d_1}F(Q)$ is. Exchanging the two summation indices replaces
$(a,b,d,m)$ by $(d,mb,a,\frac1m)$. Hence every rank-two symmetrizable Nahm sum with rational data is, up to the
substitution $\tau\mapsto d_1\tau$ and the order of the indices, a sum $q^cF$ with $m\ge1$ rational.
\end{lemma}
\begin{proof}
With $D=\operatorname{diag}(d_1,d_2)$, $\frac12n^{\mathsf T}ADn=d_1\bigl(\frac a2n_1^2+mbn_1n_2+\frac{md}2n_2^2\bigr)$ and
$(q^{d_2};q^{d_2})_{n_2}=(Q^m;Q^m)_{n_2}$. The substitution $\tau\mapsto d_1\tau$ preserves invariance under a subgroup
of finite index, because $\delta^{-1}\mathrm{SL}_2(\mathbb Z)\delta$ and $\mathrm{SL}_2(\mathbb Z)$ are commensurable for
$\delta=\operatorname{diag}(d_1,1)$, and it preserves meromorphy at the cusps. The Nahm equations involve $A$ only.
For the exchange, the matrix $\bigl(\begin{smallmatrix}d&mb\\ b&a\end{smallmatrix}\bigr)$ has lower-left entry
$\frac1m\cdot mb$.
\end{proof}
\noindent\emph{The radial expansion for rational $m$.} For $m=d_2/d_1$ the radial expansion of $F$ used below follows
from Mizuno's integer-step formula \cite[Thm.~2.1]{Miz} applied to $f$: if $Q=e^{-\varepsilon}$ then $q=e^{-\varepsilon/d_1}$,
so the radial parameter in the integer-step formula is $\varepsilon/d_1$, and the steps $d_1\varepsilon/d_1=\varepsilon$ and
$d_2\varepsilon/d_1=m\varepsilon$ give exactly the formula for the symmetrizer $\operatorname{diag}(1,m)$, with the
normalized shifts $B=b_1/d_1$, $C=b_2/d_1$. The expansion is therefore valid for every positive rational $m$; after
clearing denominators, the numerators $\Psi_j$ defined below are polynomial in the parameters.

\paragraph{Modularity.} We say that $q^cF$ is \emph{modular} if $f(\tau):=e^{2\pi ic\tau}F(e^{2\pi i\tau})$ is invariant
under a finite-index subgroup $\Gamma\subset\mathrm{SL}_2(\mathbb Z)$ (a finite-order multiplier is allowed; it
disappears after raising $f$ to a power, which does not affect what follows) and is meromorphic at every cusp:
for each $\gamma\in\mathrm{SL}_2(\mathbb Z)$, $f(\gamma\tau)$ has an expansion $\sum_{n\ge n_0}a_ne^{2\pi in\tau/h}$ with
finitely many negative-order terms. ($F$ is holomorphic on the upper half-plane.)

\paragraph{Radial expansion.} By \cite[Thm.~2.1]{Miz} with $\alpha=0$ (this case is also in \cite{Kan}; for
$m=1$ see \cite{GZ,VZ}), as $\varepsilon\to0^+$,
\[
\log F(e^{-\varepsilon})\sim\Lambda/\varepsilon+\log\omega+\sum_{j\ge1}\ell_j\varepsilon^j,
\]
where $\sum_j\ell_j\varepsilon^j$ is the logarithm of the formal Gaussian integral \cite[(15)--(18)]{Miz}: the
quadratic form is $\widetilde AD$ with $\widetilde A=A+\operatorname{diag}(x/(1-x),y/(1-y))$, which is the matrix
$\widetilde H$ below; the vertices are the series built from Bernoulli polynomials and $\Li_{2-r}(x),\Li_{2-r}(y)$;
the linear term comes from the shifts. \emph{Modularity forces $\ell_j=0$ for $j\ge2$} \cite[Lemma~3.1]{VZ}:
since $f(i\varepsilon/2\pi)=(f\circ S)(2\pi i/\varepsilon)$ and $f\circ S$ is invariant under a finite-index subgroup and
meromorphic at $\infty$, $f(i\varepsilon/2\pi)=a\,e^{\lambda/\varepsilon}\bigl(1+O(e^{-\mu/\varepsilon})\bigr)$ with $a\ne0$,
$\mu>0$. As $\log F(e^{-\varepsilon})=\log f(i\varepsilon/2\pi)+c\varepsilon$, comparison gives $\ell_1=c$ and $\ell_j=0$ for all
$j\ge2$.

\paragraph{Formal saddle coordinates.} Write $X=x/(1-x)$, $Y=y/(1-y)$; then $\Li_{-k}(x)=P_k(X)$ with
$P_0=X$, $P_{k+1}=X(1+X)P_k'$. With
\[
\widetilde H=\begin{pmatrix}a+X&mb\\ mb&m(d+Y)\end{pmatrix},\qquad
\boldsymbol\gamma=\bigl(B+\tfrac X2,\ C+\tfrac m2Y\bigr),\qquad \mathbf v=\widetilde H^{-1}\boldsymbol\gamma,
\]
the formal expansion gives $\ell_j$ as a polynomial in the $P_k$, the entries of $\widetilde H^{-1}$ and
$\mathbf v$, with at most $3j$ propagator factors. (We evaluate the formal integral in two independent ways:
after completing the square in the linear term, which produces the shifted leaf $\mathbf v$ (routine
\texttt{formalells}), and without completing it, through the moments $T_j$ of the Gaussian expectation and the
logarithm recurrence $\Psi_j=T_j-\frac1j\sum_{k=1}^{j-1}k\Psi_kT_{j-k}$. The two implementations agree coefficient by coefficient on the complete polynomials $\Psi_2,\dots,\Psi_5$ at six
parameter points, and on the top edges for all parameters; see Appendix~\ref{app:checks}.) Hence
\[
\Psi_j:=(\det\widetilde H)^{3j}\,\ell_j\ \in\ \Q[a,b,d,m,B,C][X,Y].
\]
At the Nahm point $\det\widetilde H\ne0$, so $\ell_j=0$ iff $\Psi_j(X_0,Y_0)=0$, with $(X_0,Y_0)=(x/(1-x),y/(1-y))$.

\paragraph{Rank-one notation.} For $\sum_nq^{An^2/2+\beta n}/(q)_n$ at formal point $X$ let
$\lambda^{(j)}_A(\beta;X)$ be its $j$-th correction and, for $j\ge2$, $N_j(A,\beta;X):=(A+X)^{3j}\lambda^{(j)}_A(\beta;X)$,
a polynomial of degree at most $3j-1$ in $X$ divisible by $X(1+X)$ (the first numerator $N_1$ is defined
separately in Section~\ref{sec:first}) (the even-order numerators vanish
identically for Euler shifts, Theorem~\ref{thm:rk1}).

\begin{theorem}[main]\label{thm:main}
Under the above hypotheses, if $q^cF$ is modular for some $c\in\Q$, then $[K:\Q]\le24$.
\end{theorem}

The polynomials $\Psi_j$ and $f_r$ are defined below. The step from the intersection count of Section~\ref{sec:count}
to $24$ in the hardest case uses a pole-order analysis of $\Psi_4$ along the determinant conic (Section~\ref{sec:pole}).

\paragraph{Scope.} Theorem~\ref{thm:main} is a degree bound, not a classification. Since the parameters range
over infinitely many rational values, a bound on $[K:\Q]$ does not by itself imply that the class contains
only finitely many modular sums.
The bound is also far from sharp: all known examples have $[K:\Q]\le4$, and
$[K:\Q]\le3$ when $m\ge2$ (Table~\ref{tab:known} and \cite{Z}), while the residual intersection count
for $(f_r,\Psi_2)$ can equal $24$ (Section~\ref{sec:count}); this does not establish a modular Nahm field of
degree $24$. Approaching these values therefore requires restrictions beyond the intersection count for
$(f_r,\Psi_2)$, for example the higher correction equations used jointly (end of Section~\ref{sec:count}), the Nahm
equations, or arithmetic constraints on the associated Bloch element. Modularity forces a positive integral
multiple of that element to be torsion \cite{CGZ,Miz}, but torsion alone cannot bound the degree: the Bloch group of a totally real field is
torsion whatever its degree \cite{CGZ}, so an arithmetic argument would have to use the special form of the
element and the Nahm relations. Section~\ref{sec:families} does this on two families where the Nahm equations
force the point onto a fixed curve, and obtains $[K:\Q]\le2$ there (at most $4$ on one $m=1$ family). The cases $m\ne1$ and $m=1$ are treated differently. For $m\ne1$ no
common component of the relevant curves meets the positive quadrant (Sections~\ref{sec:rigidity}--\ref{sec:b0}).
For $m=1$ the rigidity step (Lemma~\ref{lem:rig}) fails exactly at $s=\pm1$, where Zagier's continuous family of
symmetric modular sums lives \cite{Z}, and there a common component genuinely exists: the line $x+y=1$
through the Nahm point $(\frac12,\frac12)$. Section~\ref{sec:m1} shows by a different, computer-assisted
argument that this is the only one. For comparison, the analogous rank-one question, which shifts make the corrections vanish at the
Nahm point, reduces to common roots of the numerators $N_j$ and is settled by a single resultant computation
(Theorem~\ref{thm:rk1}), which recovers Zagier's seven modular pairs. The additional difficulty in rank two is
to exclude common components of plane curves whose Newton polygons depend on six parameters.

\paragraph{Outline of the proof.}
\begin{enumerate}
\item Determine the Newton polygons, edge polynomials and corners of $\Psi_2,\dots,\Psi_4$ (Section~\ref{sec:structure})
and of the first-coefficient polynomial $f_r$ (Section~\ref{sec:first}). Every edge of $f_r$ is a rank-one
expression, which can divide the corresponding edge of $\Psi_2$ only if the effective exponent of that edge is
$1$. A common factor of $f_r$ and $\Psi_2$ of full width would force both opposite effective exponents to equal
$1$, contradicting their strict inequality for $b\ne0$; so such a factor has bidegree at most $(2,2)$. The
decoupled case $b=0$ is treated separately.
\item Show that a common irreducible factor of $f_r,\Psi_2,\Psi_3,\Psi_4$ has no diagonal edges, using the
Bernoulli values $B_2,\dots,B_5$ at the corners (Section~\ref{sec:rigidity}), hence a rectangular Newton
polygon with explicitly constrained edges (Section~\ref{sec:shape}).
\item Exclude every such factor that meets $\{X,Y>0\}$: for $b\ne0$ by finite casework
(Section~\ref{sec:casework}), for $b=0$ by a rank-one argument (Section~\ref{sec:b0}). For $m=1$, Steps~2--3 are
replaced by a classification (Section~\ref{sec:m1}): a common factor can meet the six lines
$X,Y\in\{0,-1,\infty\}$ only at common roots of rank-one numerators, which reduces the problem to finitely many
root patterns.
\item Show that $\Psi_j$ lies in the $j$-th power of the ideal of the points where $f_r$ meets the conic
$\det\widetilde H=0$. On a component of $f_r$ through the Nahm point, $\Psi_j/(\det\widetilde H)^j$ is then a regular
section of a line bundle of degree $2j(w+h)$, whose zeros include the Galois orbit (Section~\ref{sec:count}).
\end{enumerate}

\paragraph{Computations.} Several steps are exact computations in PARI/GP with rational arithmetic, and for
$m=1$ also Gr\"obner basis computations in \textsc{Singular} \cite{Sing}: either
with all parameters as indeterminates, or as finite case analyses whose inputs and outputs are listed. A
few identities in the parameters are established by exact evaluation on grids exceeding a priori degree
bounds. Appendix~\ref{app:checks} lists every computational claim with the script that verifies it; the
scripts are available as ancillary files. The external inputs to the degree-bound proof are Mizuno's radial asymptotic theorem
\cite[Thm.~2.1]{Miz}, which expresses the coefficients of the radial expansion as a formal Gaussian
integral, and the classical theorems of Ostrowski and Inkeri.

\paragraph{Known examples.} The class covered by Theorem~\ref{thm:main} is not empty. Mizuno's numerical
search \cite[Table~1]{Miz} found fourteen rank-two matrices with symmetrizer $\operatorname{diag}(1,m)$,
$m\in\{2,3,4\}$, each with one to five shift vectors, $38$ pairs in all. Table~\ref{tab:known} lists them in our
notation, numbered as in \cite{WW}. Four were justified in \cite{Miz}, by known or new identities, and eight
more by Wang and Wang \cite{WW}. The remaining two, Examples~7 and~8, depended on the modulo-nine identities
of Kanade and Russell \cite{KR}, in Kur\c{s}ung\"oz's sum form \cite{Kur2}, and on a dual conjecture of Wang
and Wang; both were proved in September 2026 \cite{Miz26,Xia}, in preprints not yet refereed. All fourteen
satisfy the hypotheses of Theorem~\ref{thm:main}. Their Nahm fields have degree $1$, $2$ or $3$. For $m=1$ the
largest known degree is $4$: Zagier's Example~4, $A=\bigl(\begin{smallmatrix}4&1\\1&1\end{smallmatrix}\bigr)$, has
$y=1/(1+x)$ and $x^4+x^2-1=0$, and Zagier's Example~5, whose matrix is the inverse of that of Example~4, has the same Nahm field, since its Nahm
point is $(1-x,1-y)$. So the natural
conjectural bounds are $3$ for $m\ge2$ and $4$ in general, and $24$ is far from sharp. The examples form seven dual pairs
$A\leftrightarrow A^{-1}$ \cite{Miz,WW}. Taking logarithms in the Nahm equations shows that duality sends the
Nahm point $(x,y)$ to $(1-x,1-y)$, so both members of a pair have the same field.

For each of the $38$ pairs we computed the Nahm point to $600$ digits and evaluated $\Psi_2,\dots,\Psi_5$ there.
Each vanishes to working precision (relative size below $10^{-599}$), as modularity requires. This tests the
formal expansion on genuinely mixed-base sums with $m=2,3,4$, including negative $b$. Moreover
$\gcd(\Psi_2,\Psi_3)=1$ exactly in all $38$ cases, so every known Nahm point is an isolated common zero.
Examples 1--4, 7 and 8 are the members $k=1$ of the families T4, T1 and T5 of Section~\ref{sec:T}.
(Scripts \texttt{known\_examples.py}, \texttt{known\_gcd.py}.)

\begin{table}[ht]
\centering\small
\caption{Mizuno's rank-two examples with symmetrizer $\operatorname{diag}(1,m)$, in the notation of
Section~\ref{sec:setting}: $a,b,d$ are the entries of $A=\bigl(\begin{smallmatrix}a&b\\ mb&d\end{smallmatrix}\bigr)$ and
$(B,C)$ is the linear term. Numbering as in \cite{WW}; consecutive odd/even examples are dual. $K=\Q(x,y)$,
$\zeta_n=e^{2\pi i/n}$.}\label{tab:known}
\begin{tabular}{r l p{0.33\textwidth} l c l}
\toprule
& $(a,b,d,m)$ & shifts $(B,C)$ & $K$ & $[K{:}\Q]$ & modularity\\
\midrule
1 & $(2,1,2,2)$ & $(0,0),(0,1),(-1,-1),(-\frac12,0),(1,2)$ & $\Q$, $(x,y)=(\frac23,\frac34)$ & 1 & \cite{Bre,CW}\\
2 & $(1,-\frac12,1,2)$ & $(0,0),(-\frac12,1),(-\frac12,0),(-\frac12,\frac12),(0,1)$ & $\Q$, $(x,y)=(\frac13,\frac14)$ & 1 & \cite{LW}\\
3 & $(1,\frac12,1,2)$ & $(0,0),(0,1),(1,1)$ & $\Q(\zeta_7)^+$ & 3 & \cite{LW}\\
4 & $(2,-1,2,2)$ & $(0,0),(-1,2),(1,0)$ & $\Q(\zeta_7)^+$ & 3 & \cite{WW}\\
5 & $(4,2,4,3)$ & $(0,0)$ & $\Q$, $(x,y)=(\frac34,\frac89)$ & 1 & \cite{Cap,Kur}\\
6 & $(1,-\frac12,1,3)$ & $(0,0)$ & $\Q$, $(x,y)=(\frac14,\frac19)$ & 1 & \cite{WW}\\
7 & $(2,1,2,3)$ & $(0,0),(1,3),(2,3)$ & $\Q(\zeta_9)^+$ & 3 & \cite{Miz26,Xia}\\
8 & $(2,-1,2,3)$ & $(0,0),(-1,3),(1,0)$ & $\Q(\zeta_9)^+$ & 3 & \cite{Miz26,Xia}\\
9 & $(3,2,4,2)$ & $(-\frac12,0),(\frac12,2)$ & $\Q(\sqrt5)$ & 2 & \cite{Miz}\\
10 & $(1,-\frac12,\frac34,2)$ & $(-\frac12,\frac12),(-\frac12,1)$ & $\Q(\sqrt5)$ & 2 & \cite{WW}\\
11 & $(\frac32,\frac12,1,2)$ & $(-1,1),(-\frac12,0),(0,1)$ & $\Q(\sqrt5)$ & 2 & \cite{Miz}\\
12 & $(1,-\frac12,\frac32,2)$ & $(-\frac32,\frac52),(-\frac12,\frac12),(-\frac12,\frac32)$ & $\Q(\sqrt5)$ & 2 & \cite{WW}\\
13 & $(3,1,2,4)$ & $(-\frac12,0),(\frac32,4)$ & $\Q(\sqrt2)$ & 2 & \cite{Kur}\\
14 & $(1,-\frac12,\frac32,4)$ & $(-\frac12,1),(-\frac12,3)$ & $\Q(\sqrt2)$ & 2 & \cite{LW}\\
\bottomrule
\end{tabular}
\end{table}

\paragraph{Relation to Mizuno's search.} Mizuno's numerical search \cite[Section~4]{Miz} covered the rank-two
matrices $\frac1{a'}\bigl(\begin{smallmatrix}a_{11}&a_{12}\\ d_1^{-1}d_2a_{12}&a_{22}\end{smallmatrix}\bigr)$ and shifts
$\frac1{b'}(b_1,b_2)$ with integer parameters satisfying $-10\le a_{ij},b_i\le20$ and
$1\le a',b',d_i\le10$, subject to positive definiteness of $A\operatorname{diag}(d_1,d_2)$.
He suggested that the resulting candidates, together with Zagier's symmetric examples, might exhaust the modular
rank-two Nahm sums \cite[Introduction]{Miz}. The numerical search did not establish completeness beyond the tested
range. Theorem~\ref{thm:main} gives, to our knowledge, the first uniform bound on the degree of the Nahm field, and it
applies to every rank-two symmetrizer, in particular to every symmetrizer in Mizuno's search
(Lemma~\ref{lem:norm}): $[K:\Q]\le24$. This bounds the degree of the Nahm field; it does not establish finiteness of the possible fields or
matrices.

\paragraph{What is verified by computer.} Table~\ref{tab:verify} lists the steps that rely on computation; each is
an exact computation (exact polynomial identities, factorisations, resultants, Gr\"obner bases, Sturm counts or
finite-field enumeration), with interval arithmetic used only to locate or separate real roots. Appendix~\ref{app:checks}
lists the scripts; the reciprocity certificate of Section~\ref{sec:pole} and the rational-$m$ checks can be rechecked with one command (\texttt{code/verify\_all.sh}).
\begin{table}[ht]\centering\small
\begin{tabular}{p{0.24\textwidth}p{0.68\textwidth}}
\toprule
result & computation\\ \midrule
Sections~\ref{sec:structure}--\ref{sec:count} ($m\ne1$) & symbolic $\Psi_2,\dots,\Psi_5$, their edges and corners; the casework of
Section~\ref{sec:casework} (Gr\"obner bases, exact positivity); the count of Section~\ref{sec:count}\\
Section~\ref{sec:m1} ($m=1$) & exact enumeration of all root patterns on the six rank-one lines; saturated Gr\"obner
bases in \textsc{Singular}; modular certificates for the bidegree-$(2,2)$ strata\\
Section~\ref{sec:pole} & the moment identities; the reciprocity identities of Proposition~\ref{prop:reciprocity},
symbolic in all parameters for $j\le4$\\
\bottomrule
\end{tabular}
\caption{Computer-verified steps.}\label{tab:verify}
\end{table}

\section{Structure of \texorpdfstring{$\Psi_2,\dots,\Psi_5$}{Psi2..Psi5}}\label{sec:structure}
The proof of Theorem~\ref{thm:main} uses only $j\le4$; the case $j=5$ is recorded for
completeness.

\subsection{Newton polygons}
Let $P_n$ denote the octagon $[0,n]^2$ with the four unit corner triangles removed.

\begin{proposition}\label{prop:newt}
For $j=2,\dots,5$ and every specialisation of the parameters, $\operatorname{supp}\Psi_j\subseteq P_{3j}$, and the
eight corner coefficients of $P_{3j}$ are those of Proposition~\ref{prop:corners}. For $j=3,5$, rational
parameters and positive-definite matrices, all eight corners are nonzero, so $\Newt\Psi_3=P_9$ and
$\Newt\Psi_5=P_{15}$. For $j=2,4$ a corner can vanish (odd Bernoulli values, e.g.\ $B_3(\frac12)=0$), so
$\Newt\Psi_2,\Newt\Psi_4$ may be strictly smaller. Generically $\Psi_2$ and $\Psi_3$ have $45$ and $96$
monomials and total degrees $11$ and $17$.
\end{proposition}

\subsection{Corners are Bernoulli values}
\begin{proposition}[rank-one corners]\label{prop:rk1corners}
For $j=2,\dots,5$,
\[
[X^{3j-1}]N_j(A,\beta;X)=\frac{(-1)^{j+1}}{(j+1)!}B_{j+1}\!\Bigl(\beta+\frac A2\Bigr),\qquad
[X^{1}]N_j(A,\beta;X)=-\frac{A^{3j}}{(j+1)!}B_{j+1}\!\Bigl(\frac{\beta+\frac12}{A}\Bigr).
\]
\end{proposition}
For $j=5$ the coefficients are polynomials of bidegree at most $(15,6)$ in $(A,\beta)$, since the
expression has degree $15$ in the propagator and $6$ in $\mathbf v$; the identities were verified at all
$112$ points of a $16\times7$ grid.

\begin{proposition}[edges]\label{prop:edges}
Let $A_{\rm b}=\Delta/d$, $\beta_{\rm b}=B-bC/d$, $A_{\rm l}=\Delta/a$, $\beta_{\rm l}=(C-mbB/a)/m$. Then for $j=2,\dots,5$:
\begin{align*}
\Psi_j(X,0)&=(md)^{3j}\,N_j(A_{\rm b},\beta_{\rm b};X), & \Psi_j(0,Y)&=(am)^{3j}m^{j}\,N_j(A_{\rm l},\beta_{\rm l};Y), \\
[Y^{3j}]\Psi_j&=m^{3j}\,N_j(a,B;X), & [X^{3j}]\Psi_j&=m^{4j}\,N_j(d,C/m;Y).
\end{align*}

\end{proposition}
\begin{proof}[Proof of the first row]
At $Y=0$ all $P_k(Y)$ vanish and $q_0$ does not occur in $\ell_j$, so every vertex on the $y$-side
vanishes, and so does every term containing $v_2$ (it only enters through $y$-vertices). What remains is
the rank-one expansion in $t_1$ with propagator $(\widetilde H^{-1})_{11}=1/(A_{\rm b}+X)$ and shifted
leaf $v_1=(\beta_{\rm b}+X/2)/(A_{\rm b}+X)$, i.e. $\lambda^{(j)}_{A_{\rm b}}(\beta_{\rm b};X)$, times
$(\det\widetilde H)^{3j}=(md)^{3j}(A_{\rm b}+X)^{3j}$. The line $X=0$ is symmetric. (The same argument gives
$X=-1$ and $Y=-1$ with $a-1$, $d-1$ in place of $a$, $d$.)
\end{proof}
\emph{Second row.} For $2\le j\le5$, with all six parameters symbolic, $[Y^r]\Psi_j=0$ for $r>3j$ and the
whole coefficient of $Y^{3j}$ equals $m^{3j}N_j(a,B;X)$. This was certified with the unshifted recurrence,
independently of \texttt{formalells}; the latter gives the same for $j\le4$
(\texttt{sym4.py}) and, for $j=5$, as the leading term of a series in $1/Y$ (\texttt{top5b.py}). The right
edge and the degree bound in $X$ follow from the top edge by Lemma~\ref{lem:swap}: applying the top-edge
identity to the exchanged parameters gives
$[X^{3j}]\Psi_j=m^{7j}\cdot(1/m)^{3j}N_j(d,C/m;Y)=m^{4j}N_j(d,C/m;Y)$. Together with the first row this
proves Proposition~\ref{prop:newt}.

\begin{lemma}[exchange symmetry]\label{lem:swap}
Put $p=(a,b,d,m,B,C)$ and $p'=(d,\,mb,\,a,\,1/m,\,C/m,\,B/m)$. Then, as polynomials in the parameters and in
$X,Y$, $\Psi_j(p;X,Y)=m^{7j}\,\Psi_j(p';Y,X)$.
\end{lemma}
\begin{proof}
In the formal Gaussian integral \cite[(17)]{Miz} with $\alpha=0$, substitute $x\mapsto x/\sqrt m$ and
$\varepsilon=\varepsilon'/m$. The vertex series $\psi_{z_i}(x_i\varepsilon^{-1/2},d_i\varepsilon)$ become
$\psi_{z_i}(x_i\varepsilon'^{-1/2},(d_i/m)\varepsilon')$, the quadratic form $\widetilde AD$ becomes
$\widetilde A(D/m)$, and the linear term becomes $x^{\mathsf T}(\mathbf b/m)\,\varepsilon'^{1/2}$. So the integral for
$(A,\mathbf b,D)$ at $\varepsilon$ equals the one for $(A,\mathbf b/m,D/m)$ at $\varepsilon'$. Exchanging the two
coordinates turns the latter into the data of $p'$, with symmetrizer $\operatorname{diag}(1,1/m)$ and $X,Y$
exchanged; the construction is defined for any positive symmetrizer. Hence
$\ell_j(p;X,Y)=m^j\ell_j(p';Y,X)$. For $p'$, $\widetilde H'=\bigl(\begin{smallmatrix}d+Y&b\\ b&(a+X)/m\end{smallmatrix}\bigr)$,
so $\det\widetilde H'=\det\widetilde H/m^2$ and $\Psi_j(p)=(m^2\det\widetilde H')^{3j}m^j\ell_j(p')=m^{7j}\Psi_j(p')$.
The identity was also checked exactly for $j=2,3$ at $25$ random rational parameter points each
(\texttt{exchange\_check.py}).
\end{proof}

\begin{proposition}[corners of $\Psi_j$]\label{prop:corners}
Each corner of $P_{3j}$ is the first or last coefficient of one of the four edge numerators. Writing
$x_1=\beta_{\rm b}+A_{\rm b}/2$, $x_2=(\beta_{\rm b}+\frac12)/A_{\rm b}$, $x_3=\beta_{\rm l}+A_{\rm l}/2$,
$x_4=(\beta_{\rm l}+\frac12)/A_{\rm l}$, $\xi=x_5=B+a/2$, $x_6=(B+\frac12)/a$, $\eta=x_7=C/m+d/2$,
$x_8=(C/m+\frac12)/d$, the corner coefficients are, up to the prefactors of Proposition~\ref{prop:edges},
$\frac{(-1)^{j+1}}{(j+1)!}B_{j+1}(x_{\rm far})$ and $-\frac{A^{3j}}{(j+1)!}B_{j+1}(x_{\rm near})$.
For rational parameters and positive-definite matrices the corners of $\Psi_3$ and $\Psi_5$ never vanish,
since $B_4$ and $B_6$ have no rational roots and $A_{\rm b},A_{\rm l},a,d>0$.
\end{proposition}

\subsection{Rank-one classification}
\begin{theorem}[common roots]\label{thm:rk1}
Let $A>0$ and $\beta$ be rational. $N_2,N_3,N_4(A,\beta;\cdot)$ have a common root other than
$0,-1$ if and only if $(A,\beta)$ is one of Zagier's seven modular rank-one pairs; the extra common roots
are then the rank-one Nahm point and its conjugate:
$X^2-X-1$ for $(2,0),(2,1)$; $X^2+X-1$ for $(\frac12,0),(\frac12,\frac12)$; $X-1$ for $(1,0),(1,\pm\frac12)$.
Moreover $(1+X)^2$ divides all $N_j$ only when $A=1$; there $(1+X)^4$ divides all $N_j$, and for Euler
shifts $\beta\in\{0,\pm\frac12\}$ the gcd is $X(1+X)^6(X-1)$ (the even-$j$ numerators vanish identically).
\end{theorem}
\begin{proof}[Proof (computational)]
With $M_j=N_j/(X(1+X))$,
\[
\operatorname{Res}_X(M_2,M_3)=A^{15}(A-1)^{15}E(A,\beta),\qquad \operatorname{Res}_X(M_2,M_4)=A^{21}(A-1)^{21}E_2(A,\beta),
\]
(resultants in PARI's convention; exchanging the arguments of the second resultant changes its sign,
since $\deg M_2\cdot\deg M_4=27$ is odd, so the second identity holds up to sign in other conventions),
with $E$ irreducible of bidegree $(17,18)$ and $E_2$ of bidegree $(24,27)$; $\gcd(E,E_2)$ is constant, and
$U(A)=\operatorname{Res}_\beta(E,E_2)$ has degree $462$ with rational roots only at $A\in\{2,1,0,-1,\frac12\}$. At $A=2,\frac12$ the common $\beta$ are exactly the modular
shifts and the common roots are as stated (also for $N_5$). At $A=1$ one factors directly:
$N_2(1,\beta;X)=c\,\beta(4\beta^2-1)X(1+X)^4$, and $N_3,N_4,N_5$ contain $(1+X)^6,(1+X)^8,(1+X)^{10}$
with the extra factor $X-1$ in $N_3,N_5$. The multiplicity statement follows from
\[
M_2(-1)=c\,(\beta-1)(A-1)^3(A-\beta)(A-2\beta+1),
\]
and from the fact that on each of the three lines $\beta\in\{1,A,\frac{A+1}2\}$ the factor $M_3(-1)$
reduces to a power of $(A-1)$.
\end{proof}

\begin{corollary}[edge gcds]\label{cor:D}
For each edge, the non-monomial part $D$ of $\gcd_{2\le j\le4}$ of the edge polynomials is:
\begin{center}
\begin{tabular}{ll}
\toprule
effective pair $(A,\beta)$ & $D(Z)$\\
\midrule
$A\notin\{\frac12,1,2\}$, or $A\in\{\frac12,2\}$ with non-modular $\beta$ & $1+Z$\\
$A=2$, $\beta\in\{0,1\}$ & $(1+Z)(Z^2-Z-1)$\\
$A=\frac12$, $\beta\in\{0,\frac12\}$ & $(1+Z)(Z^2+Z-1)$\\
$A=1$, $\beta\notin\{0,\pm\frac12\}$ & $(1+Z)^4$\\
$A=1$, $\beta\in\{0,\pm\frac12\}$ & $(1+Z)^6(Z-1)$\\
\bottomrule
\end{tabular}
\end{center}
\end{corollary}
The table is unchanged if $N_5$ is included: the classification of Theorem~\ref{thm:rk1} uses only $M_2,M_3,M_4$,
at the special pairs $\gcd(N_2,N_3)$ already equals the entry (\texttt{fr\_divides3.py}), and for Euler pairs
$N_2\equiv N_4\equiv0$ while $N_3=cX(1+X)^6(X-1)$.

\subsection{The rank-one conic}
\begin{theorem}\label{thm:conic}
For $b\ne0$, $\kappa\ne0$, $X_\kappa=-a-\kappa b$, $Y_\kappa=-d-mb/\kappa$ (so $\det\widetilde H=0$):
\[
\kappa^{3j}\,\Psi_j(X_\kappa,Y_\kappa)=(-1)^jN_j\,m^{4j}b^{3j}\,\widehat Q(\kappa)^{2j},
\]
\[
\widehat Q(\kappa)=d(d-1)\kappa^3+mb(2d-1+b)\kappa^2+mb(2a-1+mb)\kappa+m\,a(a-1),
\]
with $N_2=\frac5{16}$, $N_3=\frac{1105}{1152}$ (checked also for $j=4$ at points, $N_4=\frac{565}{128}$).
In particular the restriction is independent of $B,C$, and on the conic $\Psi_2$ vanishes only at the
roots of $\widehat Q$.
\end{theorem}
\noindent\emph{Conic condition.} If $\Gamma$ divides $\Psi_2$ and $b\ne0$, then every intersection point of
$\Gamma$ with the conic $\det\widetilde H=0$ at a parameter $\kappa\in\C^*$ is a root of $\widehat Q$. With
$q_\Gamma(\kappa):=\kappa^{\deg_Y\Gamma}\Gamma(X_\kappa,Y_\kappa)=\sum_{i=0}^{D}c_i\kappa^i$, this is the
radical condition $\operatorname{rad}(q_\Gamma/\kappa^{v})\mid\widehat Q$, where $v$ is the order of vanishing
at $\kappa=0$. Since the $c_i$ depend on the parameters, it is applied branchwise: for each pair $v\le e$,
impose $c_0=\dots=c_{v-1}=0$, $c_{e+1}=\dots=c_D=0$, $c_vc_e\ne0$, and (if $e>v$)
$\tilde q=\sum_{i=v}^{e}c_i\kappa^{i-v}$ divides $\widehat Q^{\,e-v}$; the branch $q_\Gamma\equiv0$ means
$\Gamma$ is the conic itself, which has no point with $X,Y>0$ because $\Delta>0$. This covers repeated roots,
degree drops and roots at $\kappa=0$.

\begin{lemma}[weighted order at the conic points]\label{lem:worder}
Let $b\ne0$ and let $\kappa_p\ne0$ be a root of $\widehat Q$ of multiplicity $r$. Near
$p=(X_{\kappa_p},Y_{\kappa_p})$ use the coordinates $s=\kappa-\kappa_p$ and $t$, where
\[
X=-a-\kappa b,\qquad Y=-d-\frac{mb}{\kappa}-\frac{t}{mb\kappa}
\]
(a local isomorphism near $p$, with Jacobian $1/(m\kappa)$). Then $\det\widetilde H=t$, and for every $j\ge2$
the polynomial $\Psi_j$ has weighted order at least $2rj$ for the weights $(\operatorname{wt}s,\operatorname{wt}t)=(1,2r)$.
\end{lemma}
\begin{proof}
Direct computation gives $\det\widetilde H=t$ and
\[
\operatorname{adj}\widetilde H=-\frac b\kappa\,uu^{\mathsf T}-\frac{t}{b\kappa}\,e_1e_1^{\mathsf T},\qquad u=(m,\kappa)^{\mathsf T}.
\]
The cubic Taylor term of the potential is $-[X(1+X)\xi^3+m^2Y(1+Y)\eta^3]$, and on $u$ it equals
$V_3(u):=-[m^3X(1+X)+m^2\kappa^3Y(1+Y)]$, which is $-m^2\widehat Q(\kappa)$ at $t=0$.

By Wick's theorem $\ell_j$ is a sum over connected graphs $\Gamma$ whose vertices carry the terms of the formal
integral \cite[(15)--(18)]{Miz}. A vertex $v$ of valence $r_v$ carries $\varepsilon^{(r_v-2)/2+\delta_v}$ with
$\delta_v\in\mathbb Z_{\ge0}$, and $\delta_v=0$ exactly for the Taylor terms of the potential, which have $r_v\ge3$.
Edges carry $\widetilde H^{-1}=\operatorname{adj}\widetilde H/t$. Comparing powers of $\varepsilon$ gives
$E=j+V-\sum_v\delta_v$ for the numbers $E$ of edges and $V$ of vertices.

Expand every edge into its $uu^{\mathsf T}$-part and its $e_1e_1^{\mathsf T}$-part (an \emph{$e$-edge}). A resulting
term of $\Psi_j=t^{3j}\ell_j$ with $e$ $e$-edges has the factor $t^{k_0}$, $k_0=3j-E+e$. Every potential cubic
vertex none of whose three edges is an $e$-edge contributes the factor $V_3(u)$. If $V_3$ denotes the number
of potential cubic vertices, there are at least $V_3-2e$ such vertices. Now
$2E-4j=\sum_v(4-r_v-4\delta_v)$, where each summand is at most $0$ except for potential cubic vertices, where
it is $1$. Hence $V_3\ge2E-4j$, and there are at least $2E-4j-2e=2j-2k_0$ factors $V_3(u)$.

In the coordinates $(s,t)$, $V_3(u)=-m^2\widehat Q(\kappa)+tR$ has weighted order $\ge r$, and all other
factors are analytic at $p$ because $\kappa_p\ne0$ and $b\ne0$. So each term has weighted order at least
$2rk_0+r\max(0,2j-2k_0)\ge2rj$.
\end{proof}
For $r=1$ the lemma says that $\widehat Q^{\,2j-2k}$ divides the coefficient of $t^k$ in $\Psi_j$ for $k<j$. This
was verified with all parameters symbolic for $j=2$, and exactly for $j\le5$ at three random points, where
the power of $\widehat Q$ is exactly $2j-2k$ (\texttt{conic\_order.py}, \texttt{conic\_sym.py}).

\section{The first radial coefficient}\label{sec:first}
Section~\ref{sec:setting} provides one more equation: modularity forces $\ell_1=c\in\Q$. Its polynomial has
bidegree only $(3,3)$, and its edges are rank-one expressions of the same kind as those of the $\Psi_j$.

\paragraph{The polynomial.} Let $\lambda_1:=\ell_1+\frac{1+m}{24}$ be the coefficient of $\varepsilon$ in the
logarithm of the formal Gaussian integral of \cite[(17)]{Miz} without its scalar factor
$e^{-(1+m)\varepsilon/24}$, the sum of the constants $-\varepsilon/24$ of $\log(1/(q;q)_\infty)$ and
$-m\varepsilon/24$ of $\log(1/(q^m;q^m)_\infty)$. Modularity gives
\[
\lambda_1(P)=r:=c+\tfrac{1+m}{24}\in\Q .
\]
Put $D=\det\widetilde H$ and $R_1=D^3\lambda_1$. With $U=m(d+Y)$, $V=-mb$, $W=a+X$, $g=B+X/2$, $h=C+mY/2$,
$p=X(1+X)$, $q=Y(1+Y)$, $p_2=p(1+2X)$, $q_2=q(1+2Y)$, Wick's formula gives
\begin{align}
R_1={}&-\frac{X+mY}{12}D^3+\frac{pU+m^2qW}{4}D^2+\frac{g^2U+2ghV+h^2W}{2}D^2\notag\\
&-\frac{g(pU^2+m^2qVW)+h(pUV+m^2qW^2)}{2}D-\frac{p_2U^2+m^3q_2W^2}{8}D\notag\\
&+\frac5{24}\bigl(p^2U^3+m^4q^2W^3\bigr)+m^2pq\Bigl(\frac{UVW}4+\frac{V^3}6\Bigr).
\label{eq:R1}
\end{align}
In the square-completed implementation \texttt{formalells} the term $\frac12\gamma^{\mathsf T}\widetilde H^{-1}\gamma$,
$\gamma=(g,h)$, is absorbed into the constant and must be restored; \eqref{eq:R1} was checked against it with
all six parameters symbolic (\texttt{first\_coeff.py}). For $r\in\Q$ set
\[
f_r:=R_1-rD^3 .
\]
It has rational coefficients and vanishes at $P$ and at all its conjugates.

\begin{lemma}\label{lem:first}
\begin{enumerate}
\item[(a)] $R_1$ has bidegree $(3,3)$, so $\operatorname{supp}f_r\subseteq[0,3]^2$.
\item[(b)] For $b\ne0$, on the conic $R_1(X_\kappa,Y_\kappa)=-\frac5{24}m^4b^3\,\widehat Q(\kappa)^2\kappa^{-3}$.
Hence $\gcd(f_r,D)=1$, and $f_r$ has weighted order $\ge2$ for the weights $(1,2)$ of Lemma~\ref{lem:worder}
at every simple root of $\widehat Q$.
\item[(c)] For $b=0$, $R_1(-a,Y)=\frac5{24}a^2(a-1)^2m^3(d+Y)^3$.
\item[(d)] $f_r\equiv0$ only if $b=0$ and $a=d=1$; then $x=y=\frac12$ and $K=\Q$.
\item[(e)] The corners: $[X^3Y^3]f_r=m^3(\frac{m+1}{24}-r)$, which vanishes iff $c=0$;
$[X^3]f_r=d^3m^3(\frac1{24}+\frac{C^2}{2dm}-r)$; $[Y^3]f_r=a^3m^3(\frac m{24}+\frac{B^2}{2a}-r)$;
$[1]f_r=m^3\Delta^3\bigl(\frac{mdB^2-2mbBC+aC^2}{2m\Delta}-r\bigr)$.
\end{enumerate}
\end{lemma}
\begin{proof}
(a), (b), (c), (e) are read off from \eqref{eq:R1} (\texttt{first\_coeff.py}, \texttt{first\_corners.py}); in (b)
the $t$-linear terms have weight $2$ and $rD^3=rt^3$ has weight $6$. (d) If $b\ne0$, (b) shows
$f_r\not\equiv0$ unless $\widehat Q\equiv0$, which with $a,d>0$ forces $a=d=1$ and then $b=-1$, $mb=-1$,
so $m=1$ and $\Delta=1-1=0$, contradicting $\Delta>0$. If $b=0$, (c) forces $a=1$, and the exchange symmetry (Lemma~\ref{lem:swap}) forces
$d=1$; the Nahm equations then give $x=y=\frac12$.
\end{proof}

\paragraph{Edges of $f_r$.} Let $\lambda^{(1)}_A(\beta;Z)$ be the first coefficient of the rank-one formal integral
without its factor $e^{-\varepsilon/24}$, and $N_1(A,\beta;Z):=(A+Z)^3\lambda^{(1)}_A(\beta;Z)$ the rank-one first
numerator. Explicitly
\begin{multline*}
24N_1=Z^3+\bigl(3A^2+(12\beta-3)A+12\beta^2-12\beta+2\bigr)Z^2\\
+\bigl(-2A^3+(12\beta+6)A^2+(24\beta^2-12\beta-3)A\bigr)Z+12A^2\beta^2,
\end{multline*}
one has, with all parameters symbolic (\texttt{fr\_edges.py}, \texttt{fr\_divides.py}),
\begin{align*}
[Y^0]f_r&=(md)^3\bigl[N_1(A_{\rm b},\beta_{\rm b};X)-\rho_{\rm b}(A_{\rm b}+X)^3\bigr], & \rho_{\rm b}&=r-\tfrac{C^2}{2md},\\
[Y^3]f_r&=m^3\bigl[N_1(a,B;X)-\rho_{\rm t}(a+X)^3\bigr], & \rho_{\rm t}&=r-\tfrac m{24},\\
[X^0]f_r&=(am)^3m\bigl[N_1(A_{\rm l},\beta_{\rm l};Y)-\rho_{\rm l}(A_{\rm l}+Y)^3\bigr], & \rho_{\rm l}&=\tfrac rm-\tfrac{B^2}{2am},\\
[X^3]f_r&=m^4\bigl[N_1(d,C/m;Y)-\rho_{\rm r}(d+Y)^3\bigr], & \rho_{\rm r}&=\tfrac{r-1/24}{m}.
\end{align*}
(The corner terms are the values of the rank-one $\lambda_1$ at $Z=0$ and $Z=\infty$: $N_1(A,\beta;0)=\frac12A^2\beta^2$
and $[Z^3]N_1=\frac1{24}$.)

\begin{lemma}[rank-one edge divisibility]\label{lem:edgediv}
Let $A\in\Q_{>0}$, $\beta\in\Q$, $\rho\in\Q$. The polynomial $N_1(A,\beta;Z)-\rho(A+Z)^3$ is nonzero, and it divides
$N_2(A,\beta;Z)$ only if $A=1$.
\end{lemma}
\begin{proof}
Write $H=N_1-\rho(A+Z)^3$ and $N_2=Z(1+Z)M_2$ with $M_2$ cubic. $H$ has degree $3$ unless $\rho=\frac1{24}$, and
$H(0)=0$ iff $\rho=\beta^2/(2A)$; if both hold then $A=12\beta^2$. For $A\ne1$ we treat four cases, each an
elimination in $(A,\beta)$ (\texttt{fr\_divides2.py}, \texttt{fr\_divides3.py}, \texttt{fr\_divides4.py}).
(i) $\deg H=3$, $H(0)\ne0$: then $H\mid(1+Z)M_2$, so either $H=cM_2$, or $H(-1)=0$ and $H/(1+Z)$ divides $M_2$.
In the first case the $4\times3$ matrix of the coefficients of $(A+Z)^3$, $M_2$, $N_1$ has rank $\le2$; its
$3\times3$ minors have the common factor $A^2(A-1)^2$, and the resultant in $\beta$ of two of them, divided by it,
has rational roots only at $A\in\{0,1\}$. In the second case
$\rho=N_1(A,\beta;-1)/(A-1)^3$, and the remainder of $M_2$ modulo $H/(1+Z)$ gives two equations whose resultant
in $\beta$ has positive rational roots only at $A=1$ and $A=\frac23$; at $A=\frac23$ the unique solution
$\beta=\frac23$ has $\rho=\frac1{24}$, outside this case. (ii) $\rho=\frac1{24}$: $H$ is a quadratic, whose leading
coefficient vanishes only for $A=1-2\beta\pm\frac{\sqrt3}3$, and the two coefficients of the remainder of
$Z(1+Z)M_2$ modulo $H$ have a resultant with rational roots only at $A\in\{0,1\}$. (iii) $\rho=\beta^2/(2A)$,
$\rho\ne\frac1{24}$: $H=Zq$ with $q$ quadratic, and the same elimination for $q\mid(1+Z)M_2$ gives positive
rational roots $A\in\{1,3,\frac43,\frac13,\frac34,12\}$, of which all but $A=1$ lie on $A=12\beta^2$, where the
leading coefficient of $q$ vanishes and the resulting linear polynomial does not divide $N_2$. (iv) $A=12\beta^2$
is contained in case (ii). Finally $H\not\equiv0$: for $A\ne1$, $H(-A)=N_1(A,\beta;-A)=\frac5{24}A^2(A-1)^2\ne0$, and for $A=1$,
$H=(1+Z)^2\bigl[(\frac1{24}-\rho)Z+\frac{\beta^2}2-\rho\bigr]$, which vanishes identically only if $12\beta^2=1$.
\end{proof}

\paragraph{Corners of $f_r$.} A corner of $f_r$ can vanish (Lemma~\ref{lem:first}(e)). At a vanishing corner the
two adjacent coefficients are Bernoulli values: the constant term of an edge vanishes when $\rho=\beta^2/(2A)$,
its cubic term when $\rho=\frac1{24}$, and then
\begin{equation}\label{eq:B2}
[Z]\Bigl(N_1-\tfrac{\beta^2}{2A}(A+Z)^3\Bigr)=-\tfrac{A^3}{2}B_2\Bigl(\tfrac{\beta+1/2}{A}\Bigr),\qquad
[Z^2]\Bigl(N_1-\tfrac1{24}(A+Z)^3\Bigr)=\tfrac12B_2\Bigl(\beta+\tfrac A2\Bigr).
\end{equation}
These are the case $j=1$ of Proposition~\ref{prop:corners}, with the same arguments $x_1,\dots,x_8$ and, by the
edge formulas above, with the prefactors of Proposition~\ref{prop:edges} at $j=1$. Since $B_2$ has no rational
root, both coefficients adjacent to a vanishing corner are nonzero. (Checked symbolically, and on the actual
coefficients of $f_r$ at all four corners against $\Psi_2,\Psi_3,\Psi_4$; \texttt{check32.py}.)

\begin{lemma}[width]\label{lem:width}
Let $G\in\C[X,Y]$ be a common divisor of $f_r$ and $\Psi_2$, where $f_r,\Psi_2\not\equiv0$. If $b\ne0$, or if
$b=0$ and $a,d\ne1$, then $\deg_XG\le2$ and $\deg_YG\le2$.
\end{lemma}
\begin{proof}
By Lemma~\ref{lem:edgediv} the four edge polynomials of $f_r$ are nonzero, so $f_r$ has degree $3$ in each
variable and is divisible by neither $X$ nor $Y$. Suppose $\deg_XG=3$. Then $f_r=GQ$ with $Q\in\C[Y]$, and
$Q(0)\ne0$ since $Y\nmid f_r$. So $G(X,0)$ and the leading coefficient of $G$ in $Y$ are nonzero constant
multiples of the bottom and top edges of $f_r$, that is, of $N_1(A_{\rm b},\beta_{\rm b};X)-\rho_{\rm b}(A_{\rm b}+X)^3$ and
$N_1(a,B;X)-\rho_{\rm t}(a+X)^3$. As $G\mid\Psi_2$, they divide $\Psi_2(X,0)$ and the leading coefficient of $\Psi_2$ in $Y$.
If the bottom edge of $\Psi_2$ is not identically zero, it is a nonzero multiple of $N_2(A_{\rm b},\beta_{\rm b};X)$
(Proposition~\ref{prop:edges}), and Lemma~\ref{lem:edgediv} gives $A_{\rm b}=1$; if it vanishes identically, then
$N_2(A_{\rm b},\beta_{\rm b};\cdot)\equiv0$, which by the proof of Lemma~\ref{lem:euler} again forces $A_{\rm b}=1$. In the same
way the top edge gives $a=1$. For $b\ne0$ this contradicts $A_{\rm b}=\Delta/d<a$ (Lemma~\ref{lem:ineq}); for $b=0$,
$A_{\rm b}=a\ne1$. The case $\deg_YG=3$ is the same with the left and right edges, where $A_{\rm l}=\Delta/a<d$.
\end{proof}
In particular $f_r$ itself never divides $\Psi_2$ in these cases.

\paragraph{Known examples.} For all $38$ pairs of Table~\ref{tab:known} we identified the Nahm point in an
explicit number field $K$ of degree at most $3$ and verified the Nahm equations there exactly, after clearing
the denominators of the exponents. In exact arithmetic in $K$, $r=\lambda_1(P)$ is rational and
$\gcd(f_r,\Psi_2)=1$ (\texttt{known\_exact.py}), so for all of them the count of Section~\ref{sec:count} stops at $j=2$. As an independent check of the normalization, the three Kanade--Russell sums
(Example~7) give $c=-\frac1{18},\frac5{18},\frac{11}{18}$, which are the exponents of the mod~$9$ products
$\prod(1-q^n)^{-1}$ over $n\equiv\pm1,\pm3$; $\pm2,\pm3$; $\pm3,\pm4\pmod9$: the residue pairs $\pm1,\pm2,\pm4,\pm3$
contribute $-\frac{11}{36},\frac1{36},\frac{13}{36},\frac14$, that is, $-\frac92B_2(\frac a9)$ for the pair $\pm a$.

\section{Diagonal rigidity}\label{sec:rigidity}
\begin{lemma}[Bernoulli rigidity]\label{lem:rig}
Let $u,w\in\Q$ and $\gamma,s\in\C^*$ satisfy $B_n(u)=\gamma s^{n-1}B_n(w)$ for $n=2,3,4,5$ (a relation in which
both sides vanish is allowed). Then $s^2\in\{1,\frac74,\frac47\}$. In particular these relations are impossible
for $s\in\{\pm m,\pm\frac1m\}$ with $m\in\Q_{>0}$, $m\ne1$, since $\frac74$ and $\frac47$ are not squares in $\Q$.
\end{lemma}
\begin{proof}
$B_2$ and $B_4$ have no rational roots, and the rational roots of $B_3=u(u-\frac12)(u-1)$ and
$B_5=B_3\cdot(u^2-u-\frac13)$ are $0,\frac12,1$ \cite{Ink}. If $B_3(u)=0$ then $B_3(w)=0$, so $u,w\in\{0,\frac12,1\}$, and
the relations for $n=2,4$ give $s^2=\frac{B_4(u)/B_2(u)}{B_4(w)/B_2(w)}$ with $B_4/B_2=-\frac15$ at $0,1$ and $-\frac7{20}$
at $\frac12$; so $s^2\in\{1,\frac74,\frac47\}$.

Otherwise put $x=2u-1$, $y=2w-1$, $U=x^2$, $V=y^2$. Then
\[
B_2(u)=\tfrac{3U-1}{12},\quad B_3(u)=\tfrac{x(U-1)}8,\quad B_4(u)=\tfrac{15U^2-30U+7}{240},\quad B_5(u)=\tfrac{x(U-1)(3U-7)}{96},
\]
and similarly for $w$. Since $B_3(u)\ne0$ and $x$ is rational, $U,V\notin\{0,1,\frac13,\frac73\}$, so all eight values
are nonzero. Taking ratios of the relations,
\[
s^2=\frac{3U-7}{3V-7},\qquad \frac{B_4(u)}{B_2(u)}=s^2\frac{B_4(w)}{B_2(w)},\qquad
\Bigl(\frac{B_3(u)B_2(w)}{B_2(u)B_3(w)}\Bigr)^2=s^2 .
\]
Substituting the first into the other two and clearing denominators gives $(U-V)g=0$ and $(U-V)h=0$ with
\[
g=15UV-7U-7V+7,\qquad h=27U^2V^2-18U^2V+7U^2-18UV^2+16UV-14U+7V^2-14V+7.
\]
If $U\ne V$ then $g=0$; as $15U-7=0$ is incompatible with $g=0$, $V=7(U-1)/(15U-7)$, and then
$h=112U(U-1)(3U-7)(3U-1)/(15U-7)^2\ne0$. Hence $U=V$, and $s^2=1$.
\end{proof}

\begin{theorem}[no diagonal edges]\label{thm:nodiag}
Let $m\ne1$ and let $\Gamma$ be an irreducible common factor of $f_r,\Psi_2,\Psi_3,\Psi_4$, with $f_r\not\equiv0$. Then
$\Newt\Gamma$ has no edge in a diagonal direction.
\end{theorem}
\begin{proof}
By Ostrowski \cite{Ost}, $\Newt\Gamma$ is a Minkowski summand of $\Newt f_r\subseteq[0,3]^2$, so it can have an edge
with a diagonal inner normal $\nu$ only if $\Newt f_r$ has one, that is, only if the corresponding corner of $f_r$
vanishes. The face of $f_r$ in direction $\nu$ is then the segment joining the two adjacent coefficients, both
nonzero by \eqref{eq:B2}. Write $\Psi_1:=f_r$. For $j=1,\dots,4$, $\inn_\nu\Gamma=\alpha X^\cdot Y^\cdot+\beta X^\cdot Y^\cdot$
divides $\inn_\nu\Psi_j$. If both corners of that face of $\Psi_j$ are nonzero, the corner ratio equals
$-\alpha/\beta$; by \eqref{eq:B2} for $j=1$ and Proposition~\ref{prop:corners} for $j=2,3,4$ this is the relation
of Lemma~\ref{lem:rig} with $n=j+1$ and $(u,w,s)$ equal to $(\eta,\xi,\frac1m)$ (upper right), $(x_2,x_4,m)$ (lower
left), $(x_1,x_8,-m)$ (lower right), $(x_3,x_6,-\frac1m)$ (upper left); the signs come from $(-1)^{j+1}$ at far
corners and the powers of $m$ from the edge prefactors. If exactly one corner vanishes, the face is a vertex
and $\Gamma$ cannot have the edge. Both corners can vanish only for $j=2,4$, where the relation holds trivially.
Lemma~\ref{lem:rig} with $s=\pm m^{\pm1}$ gives a contradiction.
\end{proof}
This is where the first coefficient replaces $\Psi_5$: its corners supply $B_2$, and together with $B_3,B_4,B_5$
from $\Psi_2,\Psi_3,\Psi_4$ they leave no room for a diagonal edge.

\section{Shape of a common component}\label{sec:shape}
Let $b\ne0$, $m\ne1$, and let $\Gamma$ be an irreducible (over $\C$) common factor of $f_r,\Psi_2,\Psi_3,\Psi_4$.
By Theorem~\ref{thm:nodiag}, $\Newt\Gamma$ is a rectangle $[0,w]\times[0,h]$ after a monomial shift, and
$\Gamma$ has no monomial factor. Its four edge polynomials, $g_0(X)=\Gamma(X,0)$, $g_h(X)=[Y^h]\Gamma$,
$k_0(Y)=\Gamma(0,Y)$, $k_w(Y)=[X^w]\Gamma$, divide $D_{\rm bot},D_{\rm top},D_{\rm left},D_{\rm right}$
of Corollary~\ref{cor:D} for the effective pairs
\[
\text{bottom }(\Delta/d,\ B-bC/d),\quad \text{top }(a,B),\quad \text{left }(\Delta/a,\ (C-mbB/a)/m),\quad \text{right }(d,\ C/m).
\]
\begin{lemma}[exponent inequalities]\label{lem:ineq}
For $b\ne0$: $A_{\rm bot}=\Delta/d<a=A_{\rm top}$, $A_{\rm left}=\Delta/a<d=A_{\rm right}$, and
$A_{\rm bot}A_{\rm right}=A_{\rm top}A_{\rm left}=\Delta$.
\end{lemma}
\begin{corollary}\label{cor:shape}
(i) $\Gamma$ is not a line: a line $X=X^*$ or $Y=Y^*$ through $\{X,Y>0\}$ needs the same positive Nahm root on
opposite edges, hence equal exponents there. (ii) $w\ge2$ requires bottom and top pairs both special
($A\in\{\frac12,1,2\}$ with the shifts of Corollary~\ref{cor:D}) with $A_{\rm bot}<A_{\rm top}$; $h\ge2$ likewise
for left and right. (iii) $w,h\le3$. (iv) If $w,h\ge2$ then all four pairs are special, and
Lemma~\ref{lem:ineq} leaves exactly five exponent patterns, Table~\ref{tab:T}.
\end{corollary}

\section{Casework for \texorpdfstring{$b\ne0$}{b nonzero}}\label{sec:casework}
Throughout, $\varphi=\frac{1+\sqrt5}2$ and the possible edge roots are
$-1,\ \varphi,\ \bar\varphi=-\varphi^{-1},\ \varphi^{-1},\ -\varphi,\ 1$ (Corollary~\ref{cor:D}).

\subsection{Rectangles of size \texorpdfstring{$(1,1)$}{(1,1)}}
$\Gamma=\gamma_{00}+\gamma_{10}X+\gamma_{01}Y+\gamma_{11}XY$ with edge roots $r_{\rm b},r_{\rm t},s_{\rm l},s_{\rm r}$.
Matching coefficients gives the corner relation $r_{\rm b}s_{\rm r}=r_{\rm t}s_{\rm l}$, irreducibility is
$r_{\rm b}\ne r_{\rm t}$, and then $\Gamma\propto(X-r_{\rm t})Y-s_{\rm r}(X-r_{\rm b})$. Each root $\ne-1$ forces the
type of its edge. Requiring a point with $X,Y>0$ and the constraints of Lemma~\ref{lem:ineq} leaves the
$49$ configurations of Table~\ref{tab:11}.
\begin{longtable}{cccc|cccc|l}
\caption{The $49$ admissible $(1,1)$ configurations: edge roots (bottom, top, left, right), forced exponents, and how each is excluded (conic: all branches empty; inadm.: only inadmissible solutions; T$i$: routed to family T$i$, Section~\ref{sec:T}).}\label{tab:11}\\
\toprule
$r_{\rm b}$&$r_{\rm t}$&$s_{\rm l}$&$s_{\rm r}$&$A_{\rm b}$&$A_{\rm t}$&$A_{\rm l}$&$A_{\rm r}$&excluded by\\
\midrule
\endfirsthead
\toprule
$r_{\rm b}$&$r_{\rm t}$&$s_{\rm l}$&$s_{\rm r}$&$A_{\rm b}$&$A_{\rm t}$&$A_{\rm l}$&$A_{\rm r}$&excluded by\\
\midrule
\endhead
$-1$ & $\varphi$ & $-1$ & $\varphi$ & $-$ & $2$ & $-$ & $2$ & T4 \\
$-1$ & $\varphi$ & $\varphi^{-1}$ & $-1$ & $-$ & $2$ & $\tfrac{1}{2}$ & $-$ & T3 \\
$-1$ & $\varphi^{-1}$ & $-1$ & $\varphi^{-1}$ & $-$ & $\tfrac{1}{2}$ & $-$ & $\tfrac{1}{2}$ & inadm. \\
$-1$ & $\varphi^{-1}$ & $\varphi$ & $-1$ & $-$ & $\tfrac{1}{2}$ & $2$ & $-$ & inadm. \\
$-1$ & $\varphi^{-1}$ & $-\varphi$ & $1$ & $-$ & $\tfrac{1}{2}$ & $\tfrac{1}{2}$ & $1$ & conic \\
$-1$ & $\varphi^{-1}$ & $1$ & $\bar\varphi$ & $-$ & $\tfrac{1}{2}$ & $1$ & $2$ & conic \\
$-1$ & $-\varphi$ & $\varphi^{-1}$ & $1$ & $-$ & $\tfrac{1}{2}$ & $\tfrac{1}{2}$ & $1$ & conic \\
$-1$ & $-\varphi$ & $1$ & $\varphi$ & $-$ & $\tfrac{1}{2}$ & $1$ & $2$ & conic \\
$-1$ & $1$ & $-1$ & $1$ & $-$ & $1$ & $-$ & $1$ & inadm. \\
$-1$ & $1$ & $\varphi^{-1}$ & $\bar\varphi$ & $-$ & $1$ & $\tfrac{1}{2}$ & $2$ & conic \\
$-1$ & $1$ & $-\varphi$ & $\varphi$ & $-$ & $1$ & $\tfrac{1}{2}$ & $2$ & conic \\
$-1$ & $1$ & $1$ & $-1$ & $-$ & $1$ & $1$ & $-$ & inadm. \\
$\varphi$ & $-1$ & $-1$ & $\varphi^{-1}$ & $2$ & $-$ & $-$ & $\tfrac{1}{2}$ & inadm. \\
$\varphi$ & $-1$ & $\varphi$ & $-1$ & $2$ & $-$ & $2$ & $-$ & inadm. \\
$\varphi$ & $-1$ & $-\varphi$ & $1$ & $2$ & $-$ & $\tfrac{1}{2}$ & $1$ & conic \\
$\varphi$ & $-1$ & $1$ & $\bar\varphi$ & $2$ & $-$ & $1$ & $2$ & conic \\
$\bar\varphi$ & $-1$ & $\varphi^{-1}$ & $1$ & $2$ & $-$ & $\tfrac{1}{2}$ & $1$ & conic \\
$\bar\varphi$ & $-1$ & $1$ & $\varphi$ & $2$ & $-$ & $1$ & $2$ & conic \\
$\varphi^{-1}$ & $-1$ & $-1$ & $\varphi$ & $\tfrac{1}{2}$ & $-$ & $-$ & $2$ & T2 \\
$\varphi^{-1}$ & $-1$ & $\varphi^{-1}$ & $-1$ & $\tfrac{1}{2}$ & $-$ & $\tfrac{1}{2}$ & $-$ & T1 \\
$\varphi^{-1}$ & $\varphi$ & $\varphi^{-1}$ & $\varphi$ & $\tfrac{1}{2}$ & $2$ & $\tfrac{1}{2}$ & $2$ & conic \\
$\varphi^{-1}$ & $\bar\varphi$ & $-1$ & $1$ & $\tfrac{1}{2}$ & $2$ & $-$ & $1$ & conic \\
$\varphi^{-1}$ & $\bar\varphi$ & $\varphi^{-1}$ & $\bar\varphi$ & $\tfrac{1}{2}$ & $2$ & $\tfrac{1}{2}$ & $2$ & conic \\
$\varphi^{-1}$ & $\bar\varphi$ & $-\varphi$ & $\varphi$ & $\tfrac{1}{2}$ & $2$ & $\tfrac{1}{2}$ & $2$ & conic \\
$\varphi^{-1}$ & $\bar\varphi$ & $1$ & $-1$ & $\tfrac{1}{2}$ & $2$ & $1$ & $-$ & conic \\
$\varphi^{-1}$ & $1$ & $-1$ & $-\varphi$ & $\tfrac{1}{2}$ & $1$ & $-$ & $\tfrac{1}{2}$ & conic \\
$\varphi^{-1}$ & $1$ & $\bar\varphi$ & $-1$ & $\tfrac{1}{2}$ & $1$ & $2$ & $-$ & conic \\
$\varphi^{-1}$ & $1$ & $\varphi^{-1}$ & $1$ & $\tfrac{1}{2}$ & $1$ & $\tfrac{1}{2}$ & $1$ & conic \\
$\varphi^{-1}$ & $1$ & $1$ & $\varphi$ & $\tfrac{1}{2}$ & $1$ & $1$ & $2$ & conic \\
$-\varphi$ & $\varphi$ & $-1$ & $1$ & $\tfrac{1}{2}$ & $2$ & $-$ & $1$ & conic \\
$-\varphi$ & $\varphi$ & $\varphi^{-1}$ & $\bar\varphi$ & $\tfrac{1}{2}$ & $2$ & $\tfrac{1}{2}$ & $2$ & conic \\
$-\varphi$ & $\varphi$ & $-\varphi$ & $\varphi$ & $\tfrac{1}{2}$ & $2$ & $\tfrac{1}{2}$ & $2$ & conic \\
$-\varphi$ & $\varphi$ & $1$ & $-1$ & $\tfrac{1}{2}$ & $2$ & $1$ & $-$ & conic \\
$-\varphi$ & $1$ & $-1$ & $\varphi^{-1}$ & $\tfrac{1}{2}$ & $1$ & $-$ & $\tfrac{1}{2}$ & conic \\
$-\varphi$ & $1$ & $\varphi$ & $-1$ & $\tfrac{1}{2}$ & $1$ & $2$ & $-$ & conic \\
$-\varphi$ & $1$ & $-\varphi$ & $1$ & $\tfrac{1}{2}$ & $1$ & $\tfrac{1}{2}$ & $1$ & conic \\
$-\varphi$ & $1$ & $1$ & $\bar\varphi$ & $\tfrac{1}{2}$ & $1$ & $1$ & $2$ & conic \\
$1$ & $-1$ & $-1$ & $1$ & $1$ & $-$ & $-$ & $1$ & inadm. \\
$1$ & $-1$ & $\varphi^{-1}$ & $\bar\varphi$ & $1$ & $-$ & $\tfrac{1}{2}$ & $2$ & conic \\
$1$ & $-1$ & $-\varphi$ & $\varphi$ & $1$ & $-$ & $\tfrac{1}{2}$ & $2$ & conic \\
$1$ & $-1$ & $1$ & $-1$ & $1$ & $-$ & $1$ & $-$ & inadm. \\
$1$ & $\varphi$ & $-1$ & $-\varphi$ & $1$ & $2$ & $-$ & $\tfrac{1}{2}$ & conic \\
$1$ & $\varphi$ & $\bar\varphi$ & $-1$ & $1$ & $2$ & $2$ & $-$ & conic \\
$1$ & $\varphi$ & $\varphi^{-1}$ & $1$ & $1$ & $2$ & $\tfrac{1}{2}$ & $1$ & conic \\
$1$ & $\varphi$ & $1$ & $\varphi$ & $1$ & $2$ & $1$ & $2$ & conic \\
$1$ & $\bar\varphi$ & $-1$ & $\varphi^{-1}$ & $1$ & $2$ & $-$ & $\tfrac{1}{2}$ & conic \\
$1$ & $\bar\varphi$ & $\varphi$ & $-1$ & $1$ & $2$ & $2$ & $-$ & conic \\
$1$ & $\bar\varphi$ & $-\varphi$ & $1$ & $1$ & $2$ & $\tfrac{1}{2}$ & $1$ & conic \\
$1$ & $\bar\varphi$ & $1$ & $\bar\varphi$ & $1$ & $2$ & $1$ & $2$ & conic \\
\bottomrule
\end{longtable}
\noindent For each configuration and each branch of the conic condition (all six
pairs $0\le v\le e\le2$, including the branch $q_\Gamma=c\kappa^2$, $c\ne0$, which has no nonzero finite root and is
empty in all $49$ configurations), the system
consisting of the branch equations and the forced exponent equations, split into rational and
$\sqrt5$-parts, was reduced to a lex Gr\"obner basis in $(m,d,a,b)$ after saturation by
$abdm\Delta\cdot N(c_vc_e)$ ($N$ = norm from $\Q(\sqrt5)$). The positivity filter is exact: sign changes of
$s_{\rm r}(X-r_{\rm b})/(X-r_{\rm t})$ on $(0,\infty)$ are tested between the exact critical points. Result: $37$
configurations have all branches empty. Twelve (configuration, branch) pairs are nonempty:
\begin{itemize}
\item eight are inadmissible: they force $mb^2\in\{-1,-\frac12,-\frac14\}$, or $m=\pm1$ with $b$ a root of
$b^2\mp4b+2$ or $7b^2+4b-2$ (no rational roots);
\item four are admissible and force exactly the exponent patterns and modular-edge shift constraints of the
families T4, T3, T2, T1 of Table~\ref{tab:T}:
$(-1,\varphi,-1,\varphi)$ [generic branch, $a=d=2$, $mb^2=2$],
$(-1,\varphi,\varphi^{-1},-1)$ [degree drop, $a=2,d=1$, $mb^2=1$],
$(\varphi^{-1},-1,-1,\varphi)$ [$\kappa$-factor, $a=1,d=2$, $mb^2=1$],
$(\varphi^{-1},-1,\varphi^{-1},-1)$ [no finite root, $a=d=1$, $mb^2=\frac12$].
\end{itemize}
These four are excluded by Section~\ref{sec:T}: on T1--T4, $\Psi_2$ and $\Psi_3$ have no common component
at all when $m\ne1$. (An example: $a=d=2$, $b=1$, $m=2$, $\Gamma=(X-\varphi)Y-\varphi(X+1)$ passes the conic test,
$q_\Gamma\propto(\kappa+1)(\kappa+2)$ and $\widehat Q=2(\kappa+1)^2(\kappa+2)$, but $\Psi_2,\Psi_3$ are coprime there.)

\subsection{Thin rectangles}
For $w\in\{2,3\}$, $h=1$: $\Gamma=Y\,G_1(X)-s_{\rm r}G_0(X)$ with monic $G_0\mid D_{\rm bot}$, $G_1\mid D_{\rm top}$
of degree $w$, $\gcd(G_0,G_1)=1$ (else $\Gamma$ is reducible), $s_{\rm r}$ a right-edge root and
$s_{\rm l}=s_{\rm r}G_0(0)/G_1(0)$ a left-edge root; and the transposed case. Width $\ge2$ requires
bottom and top special with $A_{\rm bot}<A_{\rm top}$, i.e. one of the pairs
$(\frac12,1),(\frac12,1_E),(\frac12,2),(1,2),(1_E,2)$ ($1_E$: $A=1$ with Euler shift). With positivity and
Lemma~\ref{lem:ineq} this leaves the $38$ configurations of Table~\ref{tab:thin}.
\begin{longtable}{c|cc|cc|cc|l}
\caption{The $38$ thin-rectangle configurations. H: $w\ge2,h=1$; V: transposed. Types of the two long edges, root multisets of $G_0,G_1$, and the two short-edge roots.}\label{tab:thin}\\
\toprule
orient.& type$_0$ & type$_1$ & roots of $G_0$ & roots of $G_1$ & $s$ & $s'$ & excluded by\\
\midrule
\endfirsthead
\toprule
orient.& type$_0$ & type$_1$ & roots of $G_0$ & roots of $G_1$ & $s$ & $s'$ & excluded by\\
\midrule
\endhead
H & $\tfrac12$ & $2$ & $-1,\varphi^{-1}$ & $\varphi,\bar\varphi$ & $-1$ & $\bar\varphi$ & conic \\
H & $\tfrac12$ & $2$ & $-1,\varphi^{-1}$ & $\varphi,\bar\varphi$ & $-\varphi$ & $-1$ & conic \\
H & $\tfrac12$ & $2$ & $-\varphi,\varphi^{-1}$ & $-1,\bar\varphi$ & $-1$ & $\varphi$ & conic \\
H & $\tfrac12$ & $2$ & $-\varphi,\varphi^{-1}$ & $-1,\bar\varphi$ & $\varphi^{-1}$ & $-1$ & conic \\
H & $\tfrac12$ & $2$ & $-\varphi,\varphi^{-1}$ & $\varphi,\bar\varphi$ & $-1$ & $-1$ & conic \\
H & $\tfrac12$ & $2$ & $-\varphi,\varphi^{-1}$ & $-1,\varphi$ & $-1$ & $\bar\varphi$ & conic \\
H & $\tfrac12$ & $2$ & $-\varphi,\varphi^{-1}$ & $-1,\varphi$ & $-\varphi$ & $-1$ & conic \\
H & $\tfrac12$ & $2$ & $-1,-\varphi$ & $\varphi,\bar\varphi$ & $-1$ & $\varphi$ & conic \\
H & $\tfrac12$ & $2$ & $-1,-\varphi$ & $\varphi,\bar\varphi$ & $\varphi^{-1}$ & $-1$ & conic \\
H & $\tfrac12$ & $1$ & $-\varphi,\varphi^{-1}$ & $-1,-1$ & $-1$ & $1$ & conic \\
H & $\tfrac12$ & $1$ & $-\varphi,\varphi^{-1}$ & $-1,-1$ & $1$ & $-1$ & conic \\
H & $\tfrac12$ & $1_E$ & $-\varphi,\varphi^{-1}$ & $-1,1$ & $-1$ & $-1$ & inadm. \\
H & $\tfrac12$ & $1_E$ & $-\varphi,\varphi^{-1}$ & $-1,-1$ & $-1$ & $1$ & conic \\
H & $\tfrac12$ & $1_E$ & $-\varphi,\varphi^{-1}$ & $-1,-1$ & $1$ & $-1$ & conic \\
H & $1$ & $2$ & $-1,-1$ & $\varphi,\bar\varphi$ & $-1$ & $1$ & conic \\
H & $1$ & $2$ & $-1,-1$ & $\varphi,\bar\varphi$ & $1$ & $-1$ & conic \\
H & $1_E$ & $2$ & $-1,1$ & $\varphi,\bar\varphi$ & $-1$ & $-1$ & inadm. \\
H & $1_E$ & $2$ & $-1,-1$ & $\varphi,\bar\varphi$ & $-1$ & $1$ & conic \\
H & $1_E$ & $2$ & $-1,-1$ & $\varphi,\bar\varphi$ & $1$ & $-1$ & conic \\
V & $\tfrac12$ & $2$ & $-1,\varphi^{-1}$ & $\varphi,\bar\varphi$ & $-1$ & $\bar\varphi$ & conic \\
V & $\tfrac12$ & $2$ & $-1,\varphi^{-1}$ & $\varphi,\bar\varphi$ & $-\varphi$ & $-1$ & conic \\
V & $\tfrac12$ & $2$ & $-\varphi,\varphi^{-1}$ & $-1,\bar\varphi$ & $-1$ & $\varphi$ & conic \\
V & $\tfrac12$ & $2$ & $-\varphi,\varphi^{-1}$ & $-1,\bar\varphi$ & $\varphi^{-1}$ & $-1$ & conic \\
V & $\tfrac12$ & $2$ & $-\varphi,\varphi^{-1}$ & $\varphi,\bar\varphi$ & $-1$ & $-1$ & conic \\
V & $\tfrac12$ & $2$ & $-\varphi,\varphi^{-1}$ & $-1,\varphi$ & $-1$ & $\bar\varphi$ & conic \\
V & $\tfrac12$ & $2$ & $-\varphi,\varphi^{-1}$ & $-1,\varphi$ & $-\varphi$ & $-1$ & conic \\
V & $\tfrac12$ & $2$ & $-1,-\varphi$ & $\varphi,\bar\varphi$ & $-1$ & $\varphi$ & conic \\
V & $\tfrac12$ & $2$ & $-1,-\varphi$ & $\varphi,\bar\varphi$ & $\varphi^{-1}$ & $-1$ & conic \\
V & $\tfrac12$ & $1$ & $-\varphi,\varphi^{-1}$ & $-1,-1$ & $-1$ & $1$ & conic \\
V & $\tfrac12$ & $1$ & $-\varphi,\varphi^{-1}$ & $-1,-1$ & $1$ & $-1$ & conic \\
V & $\tfrac12$ & $1_E$ & $-\varphi,\varphi^{-1}$ & $-1,1$ & $-1$ & $-1$ & points $P_1,P_3$ \\
V & $\tfrac12$ & $1_E$ & $-\varphi,\varphi^{-1}$ & $-1,-1$ & $-1$ & $1$ & conic \\
V & $\tfrac12$ & $1_E$ & $-\varphi,\varphi^{-1}$ & $-1,-1$ & $1$ & $-1$ & conic \\
V & $1$ & $2$ & $-1,-1$ & $\varphi,\bar\varphi$ & $-1$ & $1$ & conic \\
V & $1$ & $2$ & $-1,-1$ & $\varphi,\bar\varphi$ & $1$ & $-1$ & conic \\
V & $1_E$ & $2$ & $-1,1$ & $\varphi,\bar\varphi$ & $-1$ & $-1$ & points $P_2,P_4$ \\
V & $1_E$ & $2$ & $-1,-1$ & $\varphi,\bar\varphi$ & $-1$ & $1$ & conic \\
V & $1_E$ & $2$ & $-1,-1$ & $\varphi,\bar\varphi$ & $1$ & $-1$ & conic \\
\bottomrule
\end{longtable}
\noindent In each thin configuration the two long edges are special, so the forced
exponents fix $a$ and give $m=(A_{\rm top}-A_{\rm bot})d/b^2$ (H), resp.\ fix $d$ and give
$m=a(A_{\rm right}-A_{\rm left})/b^2$ (V), leaving two free parameters. The conic condition was applied in all
branches $(v,e)$ as in Section~\ref{sec:structure}, with saturated Gr\"obner bases. $34$ configurations have all
branches empty. The remaining four have only finitely many solutions; their rational points and admissibility:
\begin{center}\small
\begin{tabular}{lll}
\toprule
config.\ (Table~\ref{tab:thin}) & rational points $(b,\cdot)$ & admissible\\
\midrule
H, $\frac12$/$1_E$ & $(-\frac52,d{=}1)$, $(5,d{=}2)$, $(-10,d{=}-1)$ & none ($m<1$)\\
H, $1_E$/$2$ & $(-5,d{=}1)$, $(-20,d{=}-2)$, $(5,d{=}2)$ & none ($m<1$)\\
V, $\frac12$/$1_E$ & $(\frac1{20},a{=}-1)$, $(\frac15,a{=}2)$, $(-\frac15,a{=}1)$ & $P_1$, $P_3$\\
V, $1_E$/$2$ & $(\frac1{10},a{=}-2)$, $(\frac25,a{=}2)$, $(-\frac15,a{=}1)$ & $P_2$, $P_4$\\
\bottomrule
\end{tabular}
\end{center}
Here, in coordinates $(a,b,d,m)$,
\[
 P_1=(2,\tfrac15,1,25),\quad P_2=(1,-\tfrac15,2,25),\quad P_3=(1,-\tfrac15,1,\tfrac{25}2),\quad P_4=(2,\tfrac25,2,\tfrac{25}2).
\]
The horizontal points with $d>0$ have $m=\frac1{25}$ or $\frac2{25}$; they are the coordinate swaps
(Lemma~\ref{lem:norm}) of $P_1,\dots,P_4$. At $P_1,\dots,P_4$, for every shift allowed by the special edges
($6$ choices each), $\gcd(\Psi_2,\Psi_3)=1$ exactly (\texttt{pole/thin\_rational.py} checks a superset of $20$ shifts at each point). The points $P_3,P_4$, with
$m=\frac{25}2$, occur only once non-integral $m$ is allowed. The thin list itself (Table~\ref{tab:thin}) was also reproduced by a separate script with exact positivity tests.

\subsection{Rectangles of size at least \texorpdfstring{$2\times2$}{2x2}, and the families T1--T4}\label{sec:T}
\begin{table}[h]\centering
\caption{All-special exponent patterns for $b\ne0$ (Corollary~\ref{cor:shape}(iv)); $k>0$ is rational.}\label{tab:T}
\begin{tabular}{lcccccl}
\toprule
 & $a$ & $d$ & $\Delta$ & $m$ & $b$ & shifts fixed by the modular edges\\
\midrule
T1 & 1 & 1 & $\frac12$ & $2k^2$ & $\pm\frac1{2k}$ & $\beta_{\rm b},\beta_{\rm l}\in\{0,\frac12\}$ determine $B,C$\\
T2 & 1 & 2 & 1 & $k^2$ & $\pm\frac1k$ & $C/m\in\{0,1\}$, $\beta_{\rm b}\in\{0,\frac12\}$\\
T3 & 2 & 1 & 1 & $k^2$ & $\pm\frac1k$ & $B\in\{0,1\}$, $\beta_{\rm l}\in\{0,\frac12\}$\\
T4 & 2 & 2 & 2 & $2k^2$ & $\pm\frac1k$ & $B\in\{0,1\}$, $C/m\in\{0,1\}$\\
T5 & 2 & 2 & 1 & $3k^2$ & $\pm\frac1k$ & all four modular; only $B=C=0$ is consistent\\
\bottomrule
\end{tabular}
\end{table}
\noindent With $k$ symbolic and the shifts as in Table~\ref{tab:T}, the gcd over
sampled lines of $\operatorname{Res}_Y(\Psi_2,\Psi_3)$ and of $\operatorname{Res}_X(\Psi_2,\Psi_3)$, as a
polynomial in $k$, is a power of $k$ in every case except $k=\pm1$ in some T2/T3 shift choices, where $m=k^2=1$.
Hence for $m\ne1$ no member of these families has a common component of $\Psi_2,\Psi_3$. The computation treats $k$
as a variable, so it covers every rational $k$; a direct check at $k\in\{\frac13,\frac12,\frac23,\frac32,\frac53\}$, both
signs of $b$ and all admissible shifts, is consistent with this (\texttt{pole/spot\_rational\_m.py}; a diagnostic,
not a certificate).
\emph{Validity of the sampled resultants.} If $\Gamma$ is a common component with $\deg_Y\Gamma>0$, its leading
$Y$-coefficient is its top edge polynomial, whose roots lie in $\{-1,1,\pm\varphi,\pm\varphi^{-1}\}$
(Corollary~\ref{cor:D}). The sample values $x_1\in\{\frac27+\frac{7j}{11}\}$ are rational and $\ne\pm1$, so
$\Gamma(x_1,Y)$ keeps positive degree and $\operatorname{Res}_Y(\Psi_2(x_1,\cdot),\Psi_3(x_1,\cdot))$ vanishes at the
corresponding parameter; likewise for $\operatorname{Res}_X$ with $y_1\in\{\frac37+\frac{5j}{11}\}$. Thus the gcd over
samples is a multiple of every exceptional polynomial, and an empty set of admissible roots is a proof.
(The families contain genuinely modular members: at $k=1$, T4, T1 and T5 contain Examples 1 and 4, 2 and 3,
and 7 and 8 of Table~\ref{tab:known}. $\Psi_2$ and $\Psi_3$ are coprime there as well.)

\begin{corollary}\label{cor:nocomp}
For $m\ne1$ and $b\ne0$, no common component of $f_r,\Psi_2,\Psi_3,\Psi_4$ meets $\{X,Y>0\}$.
\end{corollary}

\section{The decoupled case \texorpdfstring{$b=0$}{b=0}}\label{sec:b0}
\begin{proposition}
Let $m\ne1$, $b=0$, and let $\Gamma$ be a common component of $f_r,\Psi_2,\Psi_3,\Psi_4$ through the Nahm point.
Then $K\subset\Q(\sqrt5)$.
\end{proposition}
\begin{proof}
At $b=0$ the bottom and top pairs coincide, $(a,B)$, and so do the left and right pairs, $(d,C/m)$;
$X_0,Y_0$ are the rank-one Nahm points of $a$ and $d$. Theorem~\ref{thm:nodiag} still applies. If
$\Gamma=g(X)$ is a line, $X_0$ is a positive root of $D(a,B)$, so $(a,B)$ is modular and $X_0\in\Q(\sqrt5)$;
then $\ell_j(X_0,Y_0)=\lambda_a^{(j)}(B;X_0)+m^j\lambda_d^{(j)}(C/m;Y_0)=m^j\lambda_d^{(j)}(C/m;Y_0)=0$ for
$j\le4$, so by Theorem~\ref{thm:rk1} $(d,C/m)$ is modular and $Y_0\in\Q(\sqrt5)$. Horizontal lines are
symmetric. For rectangles with $w,h\ge1$:
(i) size $(1,1)$ needs $r_{\rm b}\ne r_{\rm t}$, two distinct roots of the same $D(a,B)$, and $s_{\rm l}\ne s_{\rm r}$,
two distinct roots of $D(d,C/m)$; by Corollary~\ref{cor:D} this forces $a,d\in\{\frac12,1,2\}$;
(ii) thin rectangles ($w\ge2,h=1$ or transposed) need two coprime divisors of degree $\ge2$ of one and the same
$D$, which Corollary~\ref{cor:D} excludes ($D$ has at most three distinct roots, and in the powers of $1+Z$ only
one coprime partner, of degree $\le1$);
(iii) $w,h\ge2$ needs $\deg D(a,B)\ge2$ and $\deg D(d,C/m)\ge2$, hence $a,d\in\{\frac12,1,2\}$.
In cases (i) and (iii), $X_0$ and $Y_0$ are rank-one Nahm points for exponents in $\{\frac12,1,2\}$, so
$K\subseteq\Q(\sqrt5)$ regardless of the shifts.
\end{proof}

\section{Sixfold cancellation along the determinant conic}\label{sec:pole}
In the count of Section~\ref{sec:count}, the case in which a component $G$ of $f_r$ through $P$ divides $\Psi_2$ and
$\Psi_3$ and is cut only by $\Psi_4$ needs more than the vanishing $\Psi_j\in(D^j)$ along the determinant conic: it gives
$8(w+h)\le32$. Here we show that on such a component $\Psi_4$ vanishes to order $6$ along $D=0$, which gives $24$.
Throughout, $H=\bigl(\begin{smallmatrix}a+X&mb\\ mb&m(d+Y)\end{smallmatrix}\bigr)$, $D=\det H$, $\lambda_1$ is the first
coefficient of Section~\ref{sec:first}, and $\ell_j=\Psi_j/D^{3j}$ ($j\ge2$) are the higher logarithmic coefficients.

\subsection{A three-moment obstruction}
Let $Z$ be a standard formal Gaussian ($\EE Z^{2n}=(2n-1)!!$, odd moments zero) and
\[
 \EE\exp\bigl(usZ^3+vs^2Z^4+ws^3Z^5\bigr)=1+c_1s^2+c_2s^4+c_3s^6+O(s^8).
\]
By Wick's rule, $c_1=\frac{15}2u^2+3v$, $c_2=\frac{3465}8u^4+\frac{945}2u^2v+105uw+\frac{105}2v^2$ and
\[
 c_3=\tfrac{765765}{16}u^6+\tfrac{675675}{8}u^4v+\tfrac{45045}{2}u^3w+\tfrac{135135}{4}u^2v^2+10395uvw+\tfrac{3465}{2}v^3+\tfrac{945}{2}w^2 .
\]
\begin{lemma}\label{lem:moments}
Over a field of characteristic zero, $c_1=c_2=c_3=0$ implies $u=v=w=0$.
\end{lemma}
\begin{proof}
$c_1=0$ gives $v=-\frac52u^2$, and then $c_2=105u(w-4u^3)$. If $u\ne0$ then $w=4u^3$ and $c_3=14490u^6\ne0$. If $u=0$ then
$v=0$ and $c_3=\frac{945}2w^2$, so $w=0$.
\end{proof}

\subsection{The pole-order lemma}
Work over $\C[[\tau]]$ with valuation $\ord$. Consider a normalized formal stationary-phase expansion in two variables
whose phase and amplitude jets are regular, whose Hessian has rank one at $\tau=0$ and is nonsingular over
$\C((\tau))$, with determinant of order $e>0$.
\begin{lemma}\label{lem:poles}
If $\ell_1$ is regular at $\tau=0$ and $\ell_2=\ell_3=0$, then $\ord\ell_4\ge-6e$.
\end{lemma}
\begin{proof}
A linear change over $\C[[\tau]]$ makes one quadratic direction a unit; the formal Morse lemma in that direction (only
division by a unit is needed) and Gaussian integration in it leave a one-variable expansion with regular jets and
quadratic coefficient $h$ of order $e$. After scaling by $\sqrt\varepsilon$ the phase vertices are
$a_k(\tau)\varepsilon^{(k-2)/2}z^k$ ($k\ge3$, $a_k$ regular), with $\EE_hz^{2n}=(2n-1)!!\,h^{-n}$, and the amplitude
terms are $b_{q,k}\varepsilon^{q+k/2}z^k$ ($q,k\ge0$, not both $0$), each contributing a pole of at most $e$ per unit
of $\varepsilon$. A phase vertex of degree $k$ contributes at most $M_k=(ke-2\ord a_k)/(k-2)$ per unit of $\varepsilon$,
and $M_k\le\frac32e$ for $k\ge6$. With $M=\max\{\frac32e,M_3,M_4,M_5\}$, pairing Gaussian half-edges shows that every
contribution to $\ell_j$ has valuation at least $-jM$, and taking the logarithm preserves this.
Suppose $M>\frac32e$. After a finite ramified extension, at valuation exactly $-jM$ only the vertices of degree $3,4,5$
attaining $M$ contribute, with leading coefficients $u,v,w$ not all zero, and the leading model is exactly
$\EE\exp(usZ^3+vs^2Z^4+ws^3Z^5)$. Since $M,2M,3M>0$, regularity of $\ell_1,\ell_2,\ell_3$ forces the first three
coefficients of its logarithm, hence $c_1,c_2,c_3$, to vanish, contradicting Lemma~\ref{lem:moments}. So
$M\le\frac32e$ and $\ord\ell_4\ge-4M\ge-6e$.
\end{proof}

\subsection{Reciprocity}
Put $p=(a,b,d,m,B,C)$, $p^*=\bigl(\frac d\Delta,-\frac b\Delta,\frac a\Delta,m,\frac{dB-bC}\Delta,\frac{aC-mbB}\Delta\bigr)$
and $K_0=\frac{1+m}{24}+\frac{mdB^2-2mbBC+aC^2}{2m\Delta}$.
\begin{proposition}\label{prop:reciprocity}
As rational-function identities in $a,b,d,m,B,C,X,Y$:
$D(p^*;X^{-1},Y^{-1})=D(p;X,Y)/(\Delta XY)$, $\lambda_1(p;X,Y)+\lambda_1(p^*;X^{-1},Y^{-1})=K_0$, and
$\ell_j(p;X,Y)=(-1)^j\ell_j(p^*;X^{-1},Y^{-1})$ for $j=2,3,4$.
\end{proposition}
\begin{proof}
The first identity is immediate. For the others, the numerators $N_1=R_1$ and $N_j=\Psi_j$ are generated by the
Gaussian recurrence followed by the logarithm recurrence (script \texttt{pole/recip\_cert.gp}), with all parameters
symbolic; they agree with the numerators of Section~\ref{sec:structure}. Splitting $N_j$ into homogeneous parts of
degree $q$ in $(a,b,d,B,C)$, the transform $p\mapsto p^*$ multiplies the degree-$q$ part by $\Delta^{-q}$ after the
polynomial substitution $(a,b,d,B,C)\mapsto(d,-b,a,dB-bC,aC-mbB)$, and $X^{3j}Y^{3j}N_j(X^{-1},Y^{-1})$ is a coefficient
reversal. The cleared differences
$24m\bigl(\mathcal T_1+\Delta^4N_1\bigr)-D^3\Delta^3\bigl(12(mdB^2-2mbBC+aC^2)+m\Delta(1+m)\bigr)$ and
$\mathcal T_j-(-1)^j\Delta^{3j}N_j$ ($j=2,3,4$), where $\mathcal T_j=\Delta^{L_j}(XY)^{3j}N_j(p^*;X^{-1},Y^{-1})$ with
$L_1=7$, $L_j=6j$, are identically zero (\texttt{pole/recip\_check.py}; the $q$-degrees are $7,12,18,24$).
\end{proof}

\subsection{Sixfold cancellation}
Bars denote bihomogenization on $\mathbb P^1\times\mathbb P^1$; $\bar\Psi_4$ has bidegree $(12,12)$ and $\bar D$
bidegree $(1,1)$.
\begin{theorem}\label{thm:six}
Let $b\ne0$, and let $G$ be a reduced divisor of $f_r$ with $G\mid\Psi_2$ and $G\mid\Psi_3$, of bidegree $(w,h)$, and
$\nu:\widetilde{\mathcal C}\to\mathcal C$ the normalization of the closure of $G=0$. Then $\nu^*\bar\Psi_4/(\nu^*\bar D)^6$
is a regular section of $\nu^*\mathcal O(6,6)$. If $G$ and $\Psi_4$ have no common component, there are at most
$6(w+h)$ points of $G=\Psi_4=0$ with $D\ne0$.
\end{theorem}
\begin{proof}
At a finite point of $G$ with $D=0$, $H$ has rank one ($mb\ne0$), all jets are polynomial in $X,Y$, and on the branch
$\lambda_1=r$ and $\ell_2=\ell_3=0$; $D$ does not vanish identically on the branch (Section~\ref{sec:count}). By
Lemma~\ref{lem:poles}, $\ord\Psi_4=12\ord D+\ord\ell_4\ge6\ord D$. The points at infinity of the conic are
$(\infty,-d)$ and $(-a,\infty)$, and $(\infty,\infty)$ is not on it. Since $a,d>0$, the map $(X,Y)\mapsto(X^{-1},Y^{-1})$
sends them to finite points of the conic of $p^*$, whose off-diagonal entry is still nonzero, and by
Proposition~\ref{prop:reciprocity} the dual data again have $\lambda_1$ constant and $\ell_2=\ell_3=0$. In the chart
$u=1/X$ near $(\infty,-d)$ one has $\bar D=uD$, $\bar\Psi_4=u^{12}\Psi_4$ and $D(p^*;u,Y^{-1})=\bar D/(\Delta Y)$ with
$\Delta Y$ a unit, so Lemma~\ref{lem:poles} for the dual data gives $\ord\bar\Psi_4\ge6\ord\bar D$; the other chart is
the same. Hence the quotient has no pole above $\bar D=0$, it is a section of $\nu^*\mathcal O(6,6)$, and it is nonzero
on each component when $G$ and $\Psi_4$ share none, so it has at most $6(w+h)$ zeros; every point of $G=\Psi_4=0$ with
$D\ne0$ is one of them.
\end{proof}

\begin{proposition}\label{prop:decoupled}
If $b=0$, $a,d>0$ and $a,d\ne1$, then $f_r$ and $\Psi_2$ have no common component.
\end{proposition}
\begin{proof}
For $b=0$, $\lambda_1=L_{a,B}(X)+mL_{d,C/m}(Y)$ and $\ell_2=T_{a,B}(X)+m^2T_{d,C/m}(Y)$ with the rank-one first and
second logarithmic functions $L_{A,\beta}$, $T_{A,\beta}$, which are regular at infinity and satisfy, with
$c_A=\frac5{24}A^2(A-1)^2\ne0$,
\[
 L_{A,\beta}(Z)=\frac{c_A}{(Z+A)^3}+O\bigl((Z+A)^{-2}\bigr),\qquad
 T_{A,\beta}(Z)=\frac{36}{5}\,\frac{c_A^2}{(Z+A)^6}+O\bigl((Z+A)^{-5}\bigr).
\]
On a common component some coordinate, say $X$, is nonconstant; at a point above $X=-a$ the first equation
$\lambda_1=r$ forces $Y=-d$ with the same order $e$, and leading coefficients $\alpha+\beta=0$, $\alpha\ne0$. The second
equation then has leading coefficient $\frac{36}5(\alpha^2+\beta^2)=\frac{72}5\alpha^2\ne0$ at order $-6e$, a
contradiction. Components inside a determinant line are excluded since $c_a,c_d\ne0$.
\end{proof}
When $b=0$ and $a=1$ or $d=1$, the proof of Theorem~\ref{thm:main} gives $[K:\Q]\le3$ directly.

\section{The count: proof of Theorem~\ref{thm:main}}\label{sec:count}
\begin{lemma}[the Euler exception]\label{lem:euler}
Let $m\ge1$. Then $\Psi_2\equiv0$ if and only if $(a,b,d)=(1,0,1)$, $B\in\{0,\pm\frac12\}$ and
$C/m\in\{0,\pm\frac12\}$. In that case $x=y=\frac12$ and $K=\Q$.
\end{lemma}
\begin{proof}
Suppose $\Psi_2\equiv0$. If $b\ne0$, Theorem~\ref{thm:conic} forces $\widehat Q\equiv0$. Its cubic and constant
coefficients give $d(d-1)=a(a-1)=0$, so $a=d=1$ by positivity; the other two then give $mb(1+b)=0$ and
$mb(1+mb)=0$, so $b=-1$ and $m=1$, and then $\Delta=0$, contradicting $\Delta>0$. Hence $b=0$, and then
$\ell_2(X,Y)=\lambda^{(2)}_a(B;X)+m^2\lambda^{(2)}_d(C/m;Y)$. Both terms vanish at $0$ (each $N_2$ is divisible by
$X$), so setting $Y=0$, resp.\ $X=0$, shows that both are identically zero. Now $N_2(A,\beta;X)\equiv0$ forces
both Bernoulli corners of Proposition~\ref{prop:rk1corners} to vanish:
$u:=\beta+\frac A2$ and $v:=(\beta+\frac12)/A$ are roots of $B_3(z)=z(z-\frac12)(z-1)$, so $u,v\in\{0,\frac12,1\}$,
$A=(u+\frac12)/(v+\frac12)$ and $\beta=Av-\frac12$. Three of the nine pairs have $A=1$, $\beta\in\{0,\pm\frac12\}$. The
other six are excluded by
\begin{gather*}
N_2(\tfrac12,-\tfrac14;1)=-\tfrac5{128},\qquad N_2(\tfrac13,-\tfrac16;1)=-\tfrac{25}{729},\qquad N_2(2,-\tfrac12;1)=-\tfrac52,\\
N_2(\tfrac23,\tfrac16;1)=\tfrac{65}{1458},\qquad N_2(3,-\tfrac12;1)=-25,\qquad N_2(\tfrac32,\tfrac14;1)=\tfrac{65}{128}.
\end{gather*}
Finally $N_2(1,\beta;X)=-\frac{\beta(4\beta^2-1)}{24}X(1+X)^4$ vanishes identically exactly for $\beta\in\{0,\pm\frac12\}$.
Conversely these parameters give $\Psi_2\equiv0$. The Nahm point is $x=y=\frac12$. (All values checked in
\texttt{euler\_check.py}.)
\end{proof}

\paragraph{Bihomogenization.} On $S_0=\mathbb P^1\times\mathbb P^1$ write $\bar F$ for the bihomogenization of a polynomial $F$
in $(X,Y)$ to a prescribed bidegree: $\bar D$ has bidegree $(1,1)$, where $D=\det\widetilde H$, $\bar R_1$ and $\bar f_r$ have
bidegree $(3,3)$ (Lemma~\ref{lem:first}(a)), and $\bar\Psi_j$ has bidegree $(3j,3j)$ (Proposition~\ref{prop:newt}).
Membership in an ideal generated by sections is meant locally, in trivializations of the line bundles.

\begin{lemma}[ideal powers]\label{lem:ideal}
Suppose that $\bar R_1$ and $\bar D$ have no common component, and let $\mathcal I=(\bar R_1,\bar D)=(\bar f_r,\bar D)$.
Then $\bar\Psi_j\in\mathcal I^j$ at every point of $S_0$, for $2\le j\le4$.
\end{lemma}
\begin{proof}
\emph{Generic parameters.} Take $b\ne0$ and $\widehat Q$ of degree $3$ with three simple roots, none zero. By
Lemma~\ref{lem:first}(b), $\bar R_1$ restricted to the conic vanishes exactly at the three roots, to order two, and not
at the two points of the conic at infinity ($\kappa=0,\infty$). So $V(\mathcal I)$ consists of the three conic points,
and in the coordinates $(s,t)$ of Lemma~\ref{lem:worder}, $\mathcal I=(R_1,t)=(s^2,t)$ there. Lemma~\ref{lem:worder} gives
weighted order $\ge2j$ for the weights $(1,2)$, and a monomial $s^\alpha t^\gamma$ with $\alpha+2\gamma\ge2j$ lies in
$(s^2,t)^j$. So $\bar\Psi_j\in\mathcal I^j$.

\emph{All parameters.} Let $T\subset\C^6$ be the open set of parameters $(a,b,d,m,B,C)$ for which $\bar R_1$ and $\bar D$
have no common component; it is irreducible and contains the generic parameters above. On $S_0\times T$ the
subscheme $\mathcal Z=V(\mathcal I)$ is cut out by two equations and has finite fibres, so it is finite over $T$, a
complete intersection, and flat over $T$ \cite[Thm.~23.1]{Mat}. Since $\mathcal I$ is locally generated by a regular
sequence, $\mathcal I^\nu/\mathcal I^{\nu+1}\cong\operatorname{Sym}^\nu(\mathcal I/\mathcal I^2)$ is locally free over $\mathcal O_{\mathcal Z}$
\cite[Thm.~16.2]{Mat}, hence flat over $T$, and by induction on $j$ the sheaf $\mathcal O/\mathcal I^j$ is finite and flat
over $T$. The pushforward of $\mathcal O/\mathcal I^j\otimes\mathcal O(3j,3j)$ to $T$ is therefore locally free, and the class of
$\bar\Psi_j$ is a section of it that vanishes on a dense open subset, hence everywhere; since forming this quotient
commutes with restriction to fibres, $\bar\Psi_j\in\mathcal I^j$ on every fibre.
\end{proof}
Repeated finite conic roots are already covered by Lemma~\ref{lem:worder}; the flatness argument extends the
statement to conic points at infinity and to the decoupled case. It was also checked at a simple and at a double
conic root on the boundary of $S_0$, where the required lower bounds hold (\texttt{check32c.py}).

\begin{corollary}\label{cor:count}
Suppose that $\bar R_1$ and $\bar D$ have no common component. Let $G\in\Q[X,Y]$ be a reduced divisor of $f_r$ of
bidegree $(w,h)$, and let $2\le j\le4$ be such that $G$ and $\Psi_j$ have no common component. Then at most
$2j(w+h)$ points of $V(G)\cap V(\Psi_j)$ satisfy $D\ne0$.
\end{corollary}
\begin{proof}
Let $C\subseteq S_0$ be the closure of $V(G)$. On $C$ we have $\bar f_r=0$, so Lemma~\ref{lem:ideal} gives
$\bar\Psi_j|_C=\sigma\,\bar D^j|_C$ with $\sigma$ a regular section of $\mathcal O(2j,2j)|_C$. Since $\bar D$ and $\bar f_r$ have
no common component, $\bar D$ vanishes on no component of $C$; since $G$ and $\Psi_j$ are coprime, neither does
$\sigma$. The degree of $\mathcal O(2j,2j)$ on $C$ is $2j(w+h)$. Every point of $C$ with $\Psi_j=0\ne D$ is a zero of $\sigma$,
and pulling $\sigma$ back to the normalization of $C$ shows that there are at most $2j(w+h)$ such points.
\end{proof}

\begin{proof}[Proof of Theorem~\ref{thm:main}]
Let $P=(X_0,Y_0)$ be the Nahm point and $r=c+\frac{1+m}{24}$, so that $f_r(P)=\Psi_2(P)=\Psi_3(P)=\Psi_4(P)=0$.

\emph{Degenerate cases.} If $\Psi_2\equiv0$, Lemma~\ref{lem:euler} gives $K=\Q$. If $b=0$ and $a=1$, the decoupled
Nahm equation $1-x=x$ gives $X_0=1$. At $b=0$, $R_1=m^3(d+Y)^3N_1(a,B;X)+m^4(a+X)^3N_1(d,\frac Cm;Y)$
(\texttt{first\_coeff.py}), so $\lambda_1(P)=r$ reads $N_1(d,\frac Cm;Y_0)-\rho(d+Y_0)^3=0$ with $\rho\in\Q$. This
polynomial in $Y_0$ is nonzero of degree at most $3$ (Lemma~\ref{lem:edgediv}), so $[K:\Q]\le3$. The case $d=1$ is
symmetric.

\emph{General case.} Otherwise $\bar R_1$ and $\bar D$ have no common component: for $b\ne0$ the conic is irreducible and
$R_1$ does not vanish on it (Lemma~\ref{lem:first}(b), since $\widehat Q\not\equiv0$ when $\Delta>0$); for $b=0$, $a,d\ne1$, neither
line $X=-a$, $Y=-d$ divides $R_1$ (Lemma~\ref{lem:first}(c) and the exchanged identity). Also $f_r\not\equiv0$
(Lemma~\ref{lem:first}(d)). For $m\ne1$: if $b=0$ and a common component of $f_r,\Psi_2,\Psi_3,\Psi_4$ passes
through $P$, then $[K:\Q]\le2$ (Section~\ref{sec:b0}); otherwise no such component passes through $P$, by
Corollary~\ref{cor:nocomp} when $b\ne0$. For $m=1$: by Theorem~\ref{thm:m1}, a common component through $P$ exists
only on Zagier's family, where $P=(1,1)$ and $K=\Q$; otherwise none passes through $P$. Let $G\in\Q[X,Y]$ be an irreducible factor of $f_r$ with $G(P)=0$, and let $j\in\{2,3,4\}$ be the least
index with $G\nmid\Psi_j$; it exists, for otherwise the component of $G$ through $P$ would be a common component
of $f_r,\Psi_2,\Psi_3,\Psi_4$. As $G$ is irreducible over $\Q$ and $\Psi_j$ has rational coefficients, $G$ and $\Psi_j$
have no common component over $\C$, so $V(G)\cap V(\Psi_j)$ is finite and $P$ is algebraic. The $[K:\Q]$
conjugates of $P$ are distinct points of $V(G)\cap V(\Psi_j)$, and
$D\ne0$ at each of them, because $D(P)>0$ and $D$ has rational coefficients. By Corollary~\ref{cor:count},
$[K:\Q]\le2j(w+h)$. If $j=2$ then $w,h\le3$ and $[K:\Q]\le24$. If $j=3$ then $G\mid\Psi_2$, so $w,h\le2$ by
Lemma~\ref{lem:width} and $[K:\Q]\le24$. If $j=4$ then $G\mid\Psi_2,\Psi_3$ and $w,h\le2$; by
Proposition~\ref{prop:decoupled} this forces $b\ne0$, and Theorem~\ref{thm:six} gives $[K:\Q]\le6(w+h)\le24$.
\end{proof}

\emph{The count is attained.} At $s_1=(\frac73,-\frac12,\frac52,3,\frac14,\frac23)$ and $r=\frac17$,
$\operatorname{Res}_Y(f_r,\Psi_2)=C_3(X)^4R_{24}(X)$, where $C_3=540X^3+3402X^2+6939X+4697$ is the cubic of the conic
points and $R_{24}$ is irreducible of degree $24$ (\texttt{first\_res.py}). So for $G=f_r$ and $j=2$ the intersection
count $24$ is attained, and the argument cannot give less for this pair; this concerns the count, not the degree of
a modular Nahm field.

\paragraph{Using the equations jointly.} At the generic point $s_1$, $r=\frac17$, the $24$ residual zeros of
$(f_r,\Psi_2)$ are not zeros of $\Psi_3$: the three equations have no common zero with $D\ne0$. For the first listed
shift of each of the fourteen matrices of Table~\ref{tab:known}, the common zeros of $f_r,\Psi_2,\Psi_3$ with $D\ne0$
are exactly the Nahm point and its conjugates, together with one rational point outside the positive quadrant
in Examples~1 and~2, namely $(1,-2)$ and $(1,-\frac12)$ (\texttt{joint\_check.py}). Equivalently, fix the parameters
and let $Z$ be the finite set of common zeros of $\Psi_2$ and $\Psi_3$ with $D\ne0$ and $X\notin\{0,-1\}$; each $Q\in Z$
determines $r=\lambda_1(Q)$, and a modular case is a rational root of $\Phi(r)=\prod_{Q\in Z}(r-\lambda_1(Q))$ of
multiplicity at least $[K:\Q]$. At $s_1$ the $63$ values $\lambda_1(Q)$ are distinct and irrational. In Example~7
with $(B,C)=(0,0)$, $\Phi$ has the root $\frac19$ of multiplicity $3$ from the Nahm orbit, the root $\frac1{60}$ of
multiplicity $2$ from two points on the line $Y=0$, and $48$ further distinct irrational roots
(\texttt{phi\_poly.py}). So a field of degree $k\ge2$ requires $\lambda_1$ to be constant and rational on a Galois
orbit of $k$ points of $Z$. This suggests that the
correction equations used jointly cut the count down to the conjectural bounds ($3$ for $m\ge2$, $4$ for $m=1$), but a proof would require a
parameter-uniform elimination for $f_r,\Psi_2,\Psi_3,\Psi_4$ covering every exceptional locus.

\section{The symmetric case $m=1$}\label{sec:m1}
For $m=1$ Lemma~\ref{lem:rig} allows diagonal edges ($s=\pm1$), and common components through the positive
quadrant really occur. We replace Sections~\ref{sec:rigidity}--\ref{sec:b0} by a classification. Throughout this
section $m=1$; by Lemma~\ref{lem:width} it suffices to consider rational irreducible $G$ of bidegree $(w,h)$ with
$w,h\le2$, dividing $f_r,\Psi_2,\Psi_3,\Psi_4$ (for $b=0$ we assume $a,d\ne1$; the case $b=0$, $a=1$ or $d=1$ is
treated in the proof of Theorem~\ref{thm:main}).

\begin{theorem}\label{thm:m1}
Let $m=1$, assume that $b\ne0$ or $a,d\ne1$, and let $G\in\Q[X,Y]$ be irreducible, dividing $f_r,\Psi_2,\Psi_3,\Psi_4$,
with $V(G)$ meeting $\{X,Y>0\}$. Then $d=a$, $b=1-a$, either $C=-B$ or $(B,C)=(-\frac12,-\frac12)$, and $G=XY-1$ up to a constant.
In these cases the Nahm point is $(x,y)=(\frac12,\frac12)$, that is $(X,Y)=(1,1)$, which lies on $G$.
\end{theorem}
The proof is computer-assisted and occupies the rest of this section. Conversely, $XY-1$ divides $\Psi_2$, $\Psi_3$
and $f_r$ identically on this family (checked symbolically, and for $\Psi_4$ at sample points); this is not
needed for the degree bound. Geometrically, on $x+y=1$ the sum collapses: with $a(n-k)^2/2+nk$ as the
quadratic form, summing over $n-k$ fixed gives $1/(q)_\infty$ times a theta function.

\paragraph{The six rank-one lines.} On each of the lines $Y\in\{0,\infty,-1\}$ and $X\in\{0,\infty,-1\}$ the
restriction of $\Psi_j$ is a nonzero constant times $N_j(A_L,\beta_L;\cdot)$, and that of $f_r$ a nonzero constant
times $N_1(A_L,\beta_L;\cdot)-\rho_L(A_L+\cdot)^3$, with (for $m=1$)
\begin{center}\small
\begin{tabular}{lll}
\toprule
line & $(A_L,\beta_L)$ & $\rho_L$\\
\midrule
$Y=0$ & $(\Delta/d,\ B-bC/d)$ & $r-C^2/(2d)$\\
$Y=\infty$ & $(a,\ B)$ & $r-\frac1{24}$\\
$Y=-1$ & $((\Delta-a)/(d-1),\ B-b(C-\frac12)/(d-1))$ & $r-(12C^2-12C+2d+1)/(24(d-1))$\\
$X=0$ & $(\Delta/a,\ C-bB/a)$ & $r-B^2/(2a)$\\
$X=\infty$ & $(d,\ C)$ & $r-\frac1{24}$\\
$X=-1$ & $((\Delta-d)/(a-1),\ C-b(B-\frac12)/(a-1))$ & $r-(12B^2-12B+2a+1)/(24(a-1))$\\
\bottomrule
\end{tabular}
\end{center}
(Proposition~\ref{prop:edges} and the edges of $f_r$ in Section~\ref{sec:first}; on $X=-1$ all Taylor terms of
valence $\ge2$ in $x$ vanish, since $P_k(-1)=0$ for $k\ge1$, so $x$ can be integrated out as a Gaussian, and
symmetrically on $Y=-1$; the resulting formulas were checked symbolically, \texttt{rho\_lines.py}). The lines $X=-1$, $Y=-1$ require $a\ne1$, $d\ne1$ respectively. Since $G$
divides $\Psi_2,\Psi_3,\Psi_4$ and $f_r$, the restriction of $G$ to each line divides $\gcd_{j\le4}N_j(A_L,\beta_L;\cdot)$ and
the $f_r$-restriction.

\paragraph{Common roots.} For rational $(A,\beta)$ the common roots of $N_2,N_3,N_4$ other than $0,-1$ occur only
at the pairs in the following table (the eliminant of Theorem~\ref{thm:rk1} has rational roots only at
$A\in\{2,1,0,-1,\frac12\}$, and $A\le0$ can occur only on the lines $X,Y=-1$; \texttt{special\_pairs.py}):
\begin{center}\small
\begin{tabular}{ll}
\toprule
$(A,\beta)$ & $\gcd(N_2,N_3,N_4)$ besides $Z$, $1+Z$\\
\midrule
$(2,0),(2,1)$ & $Z^2-Z-1$\\
$(\frac12,0),(\frac12,\frac12)$ & $Z^2+Z-1$\\
$(-1,\pm\frac12)$ & $Z^2+3Z+1$\\
$A=1$ & $(1+Z)^3$ more; for $\beta\in\{0,\pm\frac12\}$ also $(1+Z)^2(Z-1)$\\
$A=0$ & $Z^3$ more; for $\beta\in\{0,\frac12,1\}$ also $Z^2(Z+2)$\\
\bottomrule
\end{tabular}
\end{center}
The roots $0$ and $\infty$ are otherwise simple (the relevant corners are values of $B_4$, which has no rational
root). Hence on each line the roots of $G$ lie in $\{0,-1,\infty\}$, all simple, unless the line's pair is special.
A factor may, however, use an ordinary root once on a line whose exponent is $A=0$ or $A=1$, where the rank-one
polynomials have that root with higher multiplicity. We therefore impose every root condition as the polynomial
equation $H_L(z)=0$, where
\begin{gather*}
H_L(Z)=N_1(A_L,\beta_L;Z)-\rho_L(A_L+Z)^3,\\
H_{0,\beta,\rho}=Z^2\Bigl[(\tfrac1{24}-\rho)Z+\tfrac{6\beta^2-6\beta+1}{12}\Bigr],\qquad
H_{1,\beta,\rho}=(Z+1)^2\Bigl[(\tfrac1{24}-\rho)Z+\tfrac{\beta^2}2-\rho\Bigr],
\end{gather*}
and never through the quotients $N_1(A,\beta;0)/A^3=\beta^2/(2A)$ or $N_1(A,\beta;-1)/(A-1)^3$. For $b\ne0$ and $a,d\ne1$ a
finite line has exponent $0$ or $1$ exactly on the loci $\Delta=a$, $\Delta=d$ and $\Delta=a+d-1$, which we treat as
separate strata.

\paragraph{Root patterns.} Prescribing the roots (with multiplicity) of $G$ on the six lines is a set of linear
conditions on its coefficients, and each root $z$ on a line $L$ imposes $H_L(z)=0$ on the parameters; a special root
also imposes the corresponding special pair. A pattern is admitted when some $G$ in its solution space is nonzero on
every line; this holds exactly when no line restriction vanishes on the whole rational kernel (a vector space over
$\Q$ is not a finite union of proper subspaces). The resulting systems in $(a,b,d,B,C,r)$ are solved exactly, by
Gr\"obner bases in \textsc{Singular} \cite{Sing} saturated by $ab\,d(a-1)(d-1)\Delta$. A curve is discarded for missing the
quadrant only when all its coefficients have one sign, and as reducible only by an exact factorization or a
common factor of the whole solution space.

\begin{lemma}\label{lem:tangent}
Let $m=1$, $a,d>0$, $b\ne0$, $\Delta>0$. If $XY-1$ divides $f_r$ and is tangent to the conic $D=0$, then $a=d$ and $b=1-a$.
\end{lemma}
\begin{proof}
On $XY=1$ the conic becomes $q(X)=dX^2+(\Delta+1)X+a$. A double root means $t:=\sqrt{ad}=(\Delta+1)/2\in\Q$; with
$u=a/t$ one gets $a=tu$, $d=t/u$, $b=\pm(t-1)$, $t>\frac12$, $t\ne1$, and the double root is $X=-u$. In the
parametrization $X=-a-\kappa b$ of the conic, for $b=t-1$ the root has $\kappa=-u$ and $\widehat Q(-u)=2u(u-1)$, so the
conic condition forces $u=1$; then $X^3f_r(X,X^{-1})/(X+1)^2$ takes the value $-\frac56(t-1)^3\ne0$ at $X=-1$, so $XY-1\nmid
f_r$. For $b=1-t$ one has $\widehat Q(u)=0$, and $X^3f_r(X,X^{-1})/(X+u)^2$ takes the value $-\frac32t^2(t-1)(u-1)^2$ at
$X=-u$, forcing $u=1$, i.e.\ $a=d=t$, $b=1-t$. Both values are independent of $B,C,r$ (\texttt{m1/fix\_tangent.py}).
\end{proof}

\begin{proposition}\label{prop:ord22}
Let $m=1$, $a,d>0$, $\Delta>0$, $b\ne0$ and $a,d\ne1$. There is no $G\in\Q[X,Y]$ of bidegree $(2,2)$ dividing both $f_r$ and
$\Psi_2$ whose bihomogenization is nonzero on each of the six lines $X,Y\in\{0,\infty,-1\}$ and whose restriction to each
of them has two distinct roots in $\{0,\infty,-1\}$.
\end{proposition}
Irreducibility and positivity are not assumed.
\begin{proof}
Index $0,\infty,-1$ by $0,1,2$. Each row and column of the $3\times3$ grid contains exactly two zeros of $G$, so the
missing points are $(i,\pi(i))$ for a permutation $\pi$, written as the word $\pi(0)\pi(1)\pi(2)$. Exchanging $X,Y$
interchanges $120$ and $201$ and fixes the other four patterns. Let $V_{ij}$ be the value at the grid point $(i,j)$ of
the bidegree-$(3,3)$ homogenization of $24f_r$, with homogeneous coordinates $(1,0),(0,1),(1,-1)$ for $0,\infty,-1$.
Then $V_{ij}=p_{ij}E_{ij}$ with
\[
(p_{ij})=\begin{pmatrix}-12\Delta^2&-a^2&-(\Delta-a)^2\\-d^2&-2&-3(d-1)^2\\-(\Delta-d)^2&-3(a-1)^2&-(\Delta-a-d+1)^2\end{pmatrix}
\]
and the cores
\begin{gather*}
E_{00}=-B^2d+2BCb-C^2a+2r\Delta,\qquad E_{11}=12r-1,\\
E_{01}=-12B^2+24ar-a,\qquad E_{10}=-12C^2+24dr-d,\\
E_{12}=-4C^2+4C+8dr-d-8r,\qquad E_{21}=-4B^2+4B+8ar-a-8r,\\
E_{02}=12E_{00}+12B^2-12Bb+12aC-2\Delta-a(24r+1),\\ E_{20}=12E_{00}+12C^2-12Cb+12dB-2\Delta-d(24r+1),\\
E_{22}=E_{02}+E_{20}-12E_{00}-12(B+C)+2(a+d)+6b+24r+2
\end{gather*}
(an identity in all parameters, checked symbolically). Divisibility requires $p_{ij}E_{ij}=0$ for $j\ne\pi(i)$, and a
core may be set to zero only where its prefactor is nonzero. The prefactors vanish only on
\[
\Delta=a,\qquad \Delta=d,\qquad \Delta=a+d-1. \tag{$\ast$}
\]
We use the rank-one facts: for $A\ne0,1$, roots $\{0,\infty\}$ mean $A=12\beta^2$; roots $\{-1,\infty\}$ mean
$A=1-12(\beta-\frac12)^2$; roots $\{0,-1\}$ mean $(6\beta-3A)^2=3A(A-1)$. So for positive $A\ne1$ the pair $\{-1,\infty\}$
forces $A<1$ and $\{0,-1\}$ forces $A>1$; at $A=1$ the pair $\{0,\infty\}$ would need $\beta^2=\frac1{12}$. The exponents
are those of the six-line table.

\emph{The stratum $\Delta=a$} (and $\Delta=d$ by exchange). Here $d=1+b^2/a>1$. On $X=0$ the exponent is $1$, so the
roots include $-1$; on $X=\infty$ it is $d>1$, so they include $0$. Only $012$, $021$, $120$ remain. For $120$ the top
pair $\{-1,\infty\}$ gives $a<1$, while the exponent $(a-d)/(a-1)>1$ on $X=-1$ also carries $\{-1,\infty\}$. For $012$
the top and right pairs give $a,d>1$, the bottom pair gives $a/d\le1$ and the pair $\{0,\infty\}$ on $X=-1$ gives
$(a-d)/(a-1)\ge0$; so $a=d$, $b^2=a(a-1)$, and the top equation becomes $(6B-3a)^2=3b^2$, impossible. For $021$ one has
$a=12B^2$, $d=12C^2$, $r=\frac1{12}$; with $w=b-12BC$, the conditions at $(-1,0)$ and $(-1,-1)$ reduce, for $a\ne d$, to
\[
3B^2(2B-1)^2+C^2(24B^2-12B+1)=0,\qquad (3B-1)(2B+2C-1)(12BC-2B-4C+1)=0,
\]
whose resultant in $C$ is $(2B-1)^4(3B-1)^2(6B-1)^3(72B^3-36B^2+6B-1)$; the cubic has no rational root, since three
times it is $(6B-1)^3-2$. For $a=d$ one has $C=\pm B$ and the remaining factors have no admissible rational zeros. What
survives is the points $P_1,P_2,P_2^{\rm op},P_3$ below.

\emph{The stratum $\Delta=a+d-1$.} Now $b^2=(a-1)(d-1)$ and both exponents on $X,Y=-1$ equal $1$, so $021$, $120$, $201$,
$210$ remain. For $120$ the top pair gives $a<1$ and the left pair $\Delta/a>1$, so $d>1$, contradicting $b^2>0$;
$201$ is symmetric. For $021$ the conditions at $(0,-1)$ and $(-1,0)$ give $(2C-1)w=2B$ and $(2B-1)w=2C$, hence $B=C$ and
$b=\pm(12B^2-1)$; the sign $+$ gives $B=\frac14$, the sign $-$ gives $48B^3-24B^2+1=0$, whose reciprocal is Eisenstein at
$3$. This leaves $P_4$. For $210$ one has $a,d>1$ and, with $a=\lambda B^2$, $d=\lambda C^2$,
\[
2\lambda^2B^2-12\lambda B+\lambda+12=0=2\lambda^2C^2-12\lambda C+\lambda+12,
\]
so $B=C$ or $B+C=6/\lambda$. If $B=C$, then $a=d$, $b=\pm(a-1)$, and $E_{00}=0$ gives $P_5$ (for $b=1-a$ it gives
$4a^3-6a+3=0$, Eisenstein at $3$). If $S:=B+C=6/\lambda$, then
\[
BC=\tfrac{S(2S+1)}{12},\quad a+d=4S-1,\quad ad=(S+\tfrac12)^2,\quad r=\tfrac{2S+1}{24},\quad
E_{00}=-\tfrac{(2S+1)(2b-2S-1)(2b-2S+1)}{48},
\]
so $b=S-\frac12$ and $\Delta=2S$; the stratum gives $\Delta=4S-2$, hence $S=1$ and $B=C=\frac12$, a contradiction.

\emph{Off the strata $(\ast)$}, all present cores vanish. Pattern $102$: the top and right pairs give $a,d<1$, but the
bottom pair $\{0,-1\}$ needs $\Delta/d>1$, while $\Delta/d<a<1$. Pattern $120$ (and $201$): here $r=\frac1{12}$,
$a=-12B^2+12B-2$, $d=12C^2$; $E_{02}=0$ gives $b=(2B+2C-1)(12BC-4B-2C+1)/(2B(2C-1))$, and then
\[
E_{00}=-\frac{FT}{24B^2(2C-1)^2},\qquad E_{22}=\frac{HT}{2B^2(2C-1)^2}
\]
with $T=4MC^2+(-24B^2+20B-4)C+(2B-1)^2$, $M=12B^2-6B+1$, and quadratics $F,H$ in $C$; since
$\operatorname{disc}_C T=-48B^2(2B-1)^2<0$ and $\operatorname{Res}_C(F,H)=9216B^6M>0$, there is no real solution. Pattern $012$: $a,d>1$;
if $B=C$, then $E_{02}=0$ is a quadratic in $b$ whose discriminant, modulo the top conic, is $-4a^3(a-1)<0$; if $B\ne C$,
then $E_{02}-E_{20}=0$ gives $b=S-\frac12$ and $\Delta=2S$, while the bottom and left pairs $\{-1,\infty\}$ need
$\Delta<\min(a,d)\le\frac{a+d}2=2S-\frac12$. Pattern $021$ gives the family $A$ below, together with zeros that lie on
$(\ast)$. Pattern $210$ gives, for $B=C$, no solution (the eliminant has resultant $-110592a^5(a-1)(a^2-a+1)\ne0$ with the
top conic), and for $B\ne C$ the family $B$ below.

\emph{Certificates.} The survivors are
\begin{center}\small
\begin{tabular}{lcccccc}
\toprule
& $a$ & $b$ & $d$ & $B$ & $C$ & $r$\\ \midrule
$P_1$ & $\frac43$ & $-\frac23$ & $\frac43$ & $\frac13$ & $\frac13$ & $\frac1{12}$\\
$P_2$ & $\frac43$ & $\frac23$ & $\frac43$ & $\frac13$ & $-\frac13$ & $\frac1{12}$\\
$P_3$ & $\frac13$ & $-\frac13$ & $\frac43$ & $\frac16$ & $\frac13$ & $\frac1{12}$\\
$P_4$ & $\frac34$ & $-\frac14$ & $\frac34$ & $\frac14$ & $\frac14$ & $\frac1{12}$\\
$P_5$ & $\frac32$ & $\frac12$ & $\frac32$ & $\frac12$ & $\frac12$ & $\frac18$\\ \bottomrule
\end{tabular}
\end{center}
their images under exchange, and the rational families
\begin{align*}
A:&\ (a,b,d,B,C;r)=\bigl(12(\tfrac12-q)^2,\,-12q^2+6q-1,\,12q^2,\,\tfrac12-q,\,q;\,\tfrac1{12}\bigr),\\
B:&\ (a,b,d,B,C;r)=\Bigl(-\tfrac{(q-6)^2}{2h},\,-\tfrac{h+6}{2h},\,-\tfrac{q^2}{2h},\,\tfrac{q-6}{2h},\,-\tfrac{q}{2h};\,\tfrac{h-6}{24h}\Bigr),\quad h=q^2-6q+6.
\end{align*}
Let $F,P$ be $f_r,\Psi_2$ (for $B$, $h^3f_r$ and $h^6\Psi_2$), scaled by rational constants to primitive integer
polynomials. At each point, $\operatorname{Res}_Y(F(1,Y),P(1,Y))\equiv82,15,15,56,45\pmod{101}$ for $P_1,\dots,P_5$, with the full
$Y$-degrees $3,6$ retained, so $f_r$ and $\Psi_2$ have no common factor of positive $Y$-degree there. For each family let
$R_s(q)$ be the primitive part of $\operatorname{Res}_Y(F(s,Y,q),P(s,Y,q))$, $s=1,2$; both have degree $72$, also modulo $101$, and
$\operatorname{Res}_q(R_1,R_2)\equiv24$ (family $A$) and $52$ (family $B$) $\pmod{101}$. A common factor of positive $Y$-degree at some
$q_0$ would make both specialized Sylvester determinants vanish, also when the affine degree drops, so $R_1(q_0)=R_2(q_0)=0$,
which is impossible. The case analysis agrees with an independent computation of the same systems by Gr\"obner bases,
which finds exactly these points and families (\texttt{m1/fix\_ord22p.py}, \texttt{m1/fix\_points.py}); the grid identity
and all certificates are reproduced by \texttt{m1/ord22\_certificates.py}.
\end{proof}

\begin{itemize}
\item \emph{Only simple roots in $\{0,-1,\infty\}$.} Counting incidences with the $3\times3$ grid forces $w=h$; there are
six patterns of bidegree $(1,1)$ and six of bidegree $(2,2)$. For $(1,1)$, $G$ is the graph of a M\"obius map
permuting $\{0,-1,\infty\}$, and only $Y=X$ and $XY=1$ meet the quadrant. The line $\{X=Y\}$ meets the conic in two
distinct points (the discriminant is $(a-d)^2+4b^2>0$), so by Theorem~\ref{thm:conic} both are roots of $\widehat Q$;
this leaves only families with $\Delta=0$ or Zagier's matrix, where $\Psi_2(X,X)\not\equiv0$. For $XY=1$ the same
condition with distinct points leaves $d=a$, $b=1-a$ or $ad=b^2+b+1$, and Lemma~\ref{lem:tangent} covers tangency; the
diagonal corners then fix the shifts, giving exactly the family in the theorem. The six patterns of bidegree $(2,2)$ are excluded by Proposition~\ref{prop:ord22}.
\item \emph{Some special root} ($b\ne0$, $a,d\ne1$). There are $450$ patterns. The systems are inconsistent, or their
parameter ideals have no admissible rational points (circles $x^2+y^2=\frac1{12}$; conditions such as $b^2+1=12C^2$,
which has no rational solution because $3$ is not a sum of two rational squares; sign conditions; zero-dimensional
primes of degree $>1$), or finitely many rational points where $\gcd(\Psi_2,\Psi_3)=1$. Six M\"obius curves meeting the
quadrant are excluded by the exponent conditions together with the conic condition in the form $\widehat Q^2\equiv0\bmod
N_G$, which is valid also at tangency. Three curves using the root $1$ on $X=0$ and $-2$ on $X=-1$ force
$\Delta/a=1$ and $(\Delta-d)/(a-1)=0$, hence $\Delta=a=d$ and exponent $0$ on $Y=-1$, contradicting their root data there.
\item \emph{$b=0$.} All six pairs reduce to $(a,B)$ and $(d,C)$; there are $217$ patterns. The systems are
inconsistent, or become so when $\operatorname{Res}_X(G,\Psi_2)\equiv0$ is added, or (for six coefficient families) their
parameter ideals contain $a^2-a+1$. For $XY=1$ divisibility at $X=-a$ gives $(1-ad)^6N_2(a,B;-a)=0$, so $ad=1$, and
then $\Psi_2,\Psi_3$ are not both divisible.
\item \emph{$a=1$, $b\ne0$.} On $X=-1$, $\Psi_2,\Psi_3,\Psi_4$ restrict to constants times $Y(1+Y)$,
$Y(1+Y)(2Y+1)$, $Y(1+Y)(6Y^2+6Y+1)$ (\texttt{deg\_a1.py}), and $f_r$ to a polynomial of degree exactly $1$: its
$Y$-coefficient is a nonzero constant times $b^4(12z^2-12bz+2b^2)$ with $z=B-\frac12$, which would vanish only for $z/b=(3\pm\sqrt3)/6$.
So $G$ has at most one finite root there, in $\{0,-1\}$, or $-\frac12$ when $\Psi_2(-1,\cdot)\equiv0$. With this
alphabet there are $319$ patterns, all excluded (inconsistency, exact reducibility, one-sign coefficients,
divisibility, rationality, or signs: on a line, roots $\{0,-1\}$ need $A\notin(0,1)$, since the conic has discriminant
$48A(A-1)$ in $\beta$, while $A_{\rm b}=1-b^2/d\in(0,1)$). The case $d=1$ follows by the exchange symmetry, verified
for $m=1$ symbolically for $R_1,\Psi_2,\Psi_3$ (\texttt{sym\_check.py}).
\item \emph{$a=d=1$, $0<|b|<1$.} There are $31$ patterns, all excluded; the last two by the identities
\[
144(b^2-1)^2\bigl(C^2-C+\tfrac14\bigr)^2+4b^4=0,\qquad 144(b^2-1)^2\bigl(C^2-bC+\tfrac{2b^2-1}{12}\bigr)^2+4b^4=0,
\]
which are the eliminants of their parameter ideals and have no real solution with $b\ne0$.
\item \emph{Lines $X=X_0$, $Y=Y_0$.} A rational line through the quadrant needs $X_0=1$ (so $a=1$ and the bottom
pair has $A=1$, impossible since $A_{\rm b}<1$), and a pair of conjugate lines needs $a=A_{\rm b}=2$ or $\frac12$, which
gives $b=0$, where $\Psi_3$ restricts to a nonzero multiple of $N_3(d,C;\cdot)$.
\end{itemize}
The complete case table, with the parameter conditions and the reason for each exclusion, the \textsc{Singular}
inputs and outputs, and the scripts are in the ancillary directory \texttt{code/m1/} (\texttt{CASES.md}). This
completes the proof of Theorem~\ref{thm:m1}.

\section{Two exceptional families}\label{sec:families}
Section~\ref{sec:count} exhibits parameters for which the residual intersection count for $(f_r,\Psi_2)$ is $24$. On
the two families considered here, the Nahm equations force the distinguished point onto $x=y$ or $x+y=1$, and
additional constraints give much smaller field degrees. The extra input is the modularity-to-torsion theorem
\cite{CGZ,Miz}: if $q^cF$ is modular, then, after clearing denominators, a positive integral multiple of
$[x]+m^{-1}[y]$ is a torsion element of the Bloch group $B(K)$. Consequently, for every embedding
$\sigma:K\hookrightarrow\C$
\begin{equation}\label{eq:regulator}
 m\,\mathcal D(\sigma x)+\mathcal D(\sigma y)=0,
\end{equation}
where $\mathcal D$ is the Bloch--Wigner dilogarithm. We use three standard properties \cite{Z}: $\mathcal D(z)>0$
for $\operatorname{Im}z>0$, $\mathcal D(\bar z)=-\mathcal D(z)$, and $\mathcal D(1-z)=-\mathcal D(z)$. In particular
$\mathcal D(z)=0$ exactly when $z$ is real.

\begin{lemma}[totally real rank-one points]\label{lem:tr1}
Let $t\in(0,1)$ be algebraic with $1-t=t^\eta$ for some $\eta\in\Q_{>0}$, and suppose that every conjugate of $t$ is
real. Then $(\eta,t)$ is $(1,\frac12)$, $(2,\frac{\sqrt5-1}2)$ or $(\frac12,\frac{3-\sqrt5}2)$; in particular
$[\Q(t):\Q]\le2$.
\end{lemma}
\begin{proof}
Write $\eta=p/q$ in lowest terms. Every conjugate of $t$ is a root of $P(T)=T^p-(1-T)^q$, which has exactly one root
in $(0,1)$. A root $T=-v<0$ requires $p$ even and $q$ odd, and then $v^p=(1+v)^q$: there is no solution if $p\le q$,
and exactly one if $p>q$, since $v^p/(1+v)^q$ increases strictly from $0$ to $\infty$. By the symmetry
$T\mapsto1-T$, $(p,q)\mapsto(q,p)$, a root $T>1$ requires $q$ even and $q>p$, and then there is exactly one. The two
outside cases exclude each other, so $P$ has at most two distinct real roots and $[\Q(t):\Q]\le2$. If $p=q$ then
$p=q=1$ and $t=\frac12$. Otherwise $P$ has leading coefficient $\pm1$ and $|P(0)|=|P(1)|=1$, so $t$ and $1-t$ are
units, which excludes $t\in\Q$. For $[\Q(t):\Q]=2$ the minimal polynomial $T^2-uT+n$ has $n=\pm1$ and
$1-u+n=\pm1$; the only cases with a root in $(0,1)$ and real roots are $T^2+T-1$ and $T^2-3T+1$, with
$\eta=2$ and $\eta=\frac12$ respectively, since $\eta$ is determined by $t$.
\end{proof}

\begin{theorem}[equal row sums]\label{thm:equalrows}
Suppose $a+b=mb+d=:\eta$. If $q^cF$ is modular, then $\eta\in\{\frac12,1,2\}$ and $[K:\Q]\le2$.
\end{theorem}
\begin{proof}
$A$ is similar to the positive-definite matrix $\operatorname{diag}(1,m)^{-1/2}(A\operatorname{diag}(1,m))\operatorname{diag}(1,m)^{-1/2}$, so its
eigenvalue $\eta$ (eigenvector $(1,1)$) is positive, and $1-t=t^\eta$ has a unique solution $t\in(0,1)$. Since $(t,t)$
solves both Nahm equations, uniqueness of the Nahm point gives $x=y=t$. Equation \eqref{eq:regulator} becomes
$(m+1)\mathcal D(\sigma t)=0$ for every $\sigma$, so all conjugates of $t$ are real, and Lemma~\ref{lem:tr1} applies.
\end{proof}

Conversely, if $x=y=t$, the Nahm equations give $1-t=t^{a+b}=t^{mb+d}$. Since $0<t<1$, comparison of logarithms
gives $a+b=mb+d$. Thus
\[
 x=y\quad\Longleftrightarrow\quad a-d+(1-m)b=0.
\]

\begin{theorem}[complementary coordinates, $m\ne1$]\label{thm:compl}
Suppose $x+y=1$, equivalently $\Delta+(m+1)b-1=0$. If $m\ne1$ and $q^cF$ is modular, then $[K:\Q]\le2$.
\end{theorem}
\begin{proof}
If $x+y=1$, the Nahm equations read $(1-x)^{1-b}=x^a$ and $x^{1-mb}=(1-x)^d$. Both logarithms are negative and
$a,d>0$, so $1-b>0$ and $1-mb>0$; with $\eta:=a/(1-b)$ we get $1-x=x^\eta$, $a=\eta(1-b)$, $d=(1-mb)/\eta$, and then
$\Delta=ad-mb^2=1-(m+1)b$. Conversely, if $\Delta+(m+1)b-1=0$ then $b<1/(m+1)$, the same formula gives a
solution $(x,1-x)$ of the Nahm equations, which is the Nahm point by uniqueness. Now $\mathcal D(1-z)=-\mathcal D(z)$
turns \eqref{eq:regulator} into $(m-1)\mathcal D(\sigma x)=0$, so $x$ is totally real and Lemma~\ref{lem:tr1} applies
to $(\eta,x)$.
\end{proof}

At $m=1$ the regulator equation is identically satisfied on $x+y=1$, and the corrections take over. By
the parameter equivalence established in the proof of Theorem~\ref{thm:compl}, which also holds when $m=1$, the
matrices with complementary Nahm point in this case are exactly
\begin{equation}\label{eq:m1family}
 A=\begin{pmatrix}\eta\theta&1-\theta\\1-\theta&\theta/\eta\end{pmatrix},\qquad \eta>0,\ \theta>\tfrac12,\ \eta,\theta\in\Q,
\end{equation}
with $\theta=1-b$, $\Delta=2\theta-1$ and Nahm point $(t,1-t)$, $1-t=t^\eta$. In the coordinates of
Section~\ref{sec:first} this point is $(X,Y)=(Z,1/Z)$ with $Z=t/(1-t)$; Zagier's family is $\eta=1$.

\begin{theorem}[the unweighted complementary family]\label{thm:m1compl}
If $q^cF$ is modular for a matrix \eqref{eq:m1family}, then $[K:\Q]\le4$. More precisely: if $\eta=1$ the Nahm
point is $(\frac12,\frac12)$; if $\theta=1$ and $\eta\ne1$, then $[K:\Q]\le2$ when $B\ne\eta C$ and $[K:\Q]\le3$ when
$B=\eta C$.
\end{theorem}
\begin{proof}
Substituting $(a,b,d,m)=(\eta\theta,1-\theta,\theta/\eta,1)$ and $(X,Y)=(Z,1/Z)$ into $R_1$ and $D$ of
Section~\ref{sec:first} gives $D=\theta(Z+\eta)^2/(\eta Z)$ and
\[
 \lambda_1(Z,1/Z)=\frac{N(Z)}{24\theta(Z+\eta)^4},\qquad \deg_ZN\le4,\qquad
 N(-\eta)=-36\eta^3(\eta-1)^2(\theta-1),
\]
with $B,C$ symbolic (\texttt{families\_check.py}). Modularity gives $\lambda_1(Z,1/Z)=r\in\Q$, i.e.
$N(Z)-24\theta r(Z+\eta)^4=0$. If $\eta\ne1$ and $\theta\ne1$, this polynomial is nonzero at $Z=-\eta$, hence nonzero
of degree at most $4$, and $Z>0$ is a root; so $[K:\Q]=[\Q(Z):\Q]\le4$. If $\eta=1$, then $t=\frac12$.

Let $\theta=1$, $\eta\ne1$. Then
\[
 \lambda_1(Z,1/Z)=\frac{12B^2(\eta+Z)+12B(\eta-1)Z+12C^2\eta Z(\eta+Z)-12C\eta(\eta-1)Z+(Z+\eta)^2}{24(Z+\eta)^2},
\]
whose derivative in $Z$ is
$(\eta C-B)\bigl((B+\eta C)(\eta+Z)+(\eta-1)(Z-\eta)\bigr)/(2(Z+\eta)^3)$. The second factor is not identically
zero for $\eta\ne1$, since that would need $B+\eta C=\eta-1$ and $B+\eta C=1-\eta$ at once. So if $B\ne\eta C$ the
function is nonconstant, and the level equation $\lambda_1=r$ is a nonzero polynomial equation of degree at most
$2$ in $Z$. If $B=\eta C$, $\lambda_1$ is constant and we use the second correction: with $w=Z+\eta$,
\[
 \ell_2(Z,1/Z)=-\frac{\eta Z(Z+1)\,Q(Z)}{24(Z+\eta)^6},\qquad
 Q(-\eta)=-15\eta^2(\eta-1)^3\ne0,
\]
where
\begin{align*}
 Q(Z)={}&(2C+1)(2C\eta+\eta-2)(2C\eta+\eta-1)\,w^3\\
 &-(\eta-1)\bigl(24C^2\eta^2+32C\eta^2-28C\eta+11\eta^2-19\eta+6\bigr)w^2\\
 &+5\eta(\eta-1)^2(6C\eta+5\eta-4)\,w-15\eta^2(\eta-1)^3 .
\end{align*}
Since $Z(Z+1)\ne0$, modularity ($\ell_2=0$ at the Nahm point) forces $Q(Z)=0$, so $[K:\Q]\le3$.
\end{proof}

These three results use the regulator where its two terms cannot cancel, and the corrections where they do.
Together they give an explicit necessary condition in the matrix parameters.

\begin{corollary}[parameter exclusions]\label{cor:exclusions}
Suppose that $q^cF$ is modular. If $[K:\Q]\ge5$, then
\[
 \bigl(a-d+(1-m)b\bigr)\bigl(ad-mb^2+(m+1)b-1\bigr)\ne0 .
\]
Equivalently, the Nahm point satisfies $x\ne y$ and $x+y\ne1$. If $m\ne1$, the same conclusion already holds when
$[K:\Q]\ge3$.
\end{corollary}
\begin{proof}
Vanishing of the first factor is equivalent to equal row sums, hence to $x=y$, and Theorem~\ref{thm:equalrows}
gives $[K:\Q]\le2$. Vanishing of the second factor is equivalent to $x+y=1$, by the proof of
Theorem~\ref{thm:compl}; then Theorem~\ref{thm:compl} gives $[K:\Q]\le2$ when $m\ne1$, and
Theorem~\ref{thm:m1compl} gives $[K:\Q]\le4$ when $m=1$.
\end{proof}

In the coordinates $X=x/(1-x)$, $Y=y/(1-y)$ of Section~\ref{sec:first}, these exclusions read $(X-Y)(XY-1)\ne0$. A
search for modular Nahm points of degree at least five may therefore saturate its defining ideal by $(X-Y)(XY-1)$.
This removes neither such a point nor any member of its Galois orbit, since the polynomial has rational coefficients
and does not vanish at the point. The saturation does not by itself give a smaller uniform degree bound.

\section*{Statements and declarations}
\paragraph{Use of AI tools.} Large language model based assistants, including Claude (Anthropic) and ChatGPT (OpenAI), were used
extensively in this work: to write and run the PARI/GP and Python scripts, to carry out and cross-check the
symbolic computations, to draft and edit the text, and to check the arguments. Several steps of the proof were
first proposed in AI-assisted notes and then verified, among them the use of the first radial coefficient
(Section~\ref{sec:first}), the Bernoulli rigidity argument with $B_2,\dots,B_5$ (Section~\ref{sec:rigidity}), the
width lemma (Lemma~\ref{lem:width}) and the ideal-power argument (Lemma~\ref{lem:ideal}). The classification for
$m=1$ (Section~\ref{sec:m1}), including Lemma~\ref{lem:tangent} and Proposition~\ref{prop:ord22}, was carried out with
AI assistance throughout, as was the pole-order argument of Section~\ref{sec:pole}, first
proposed in an AI-assisted proof supplement, whose computations were redone independently. The author directed the
work and takes full responsibility for its content. Every computational claim can be reproduced with the
ancillary files. No AI tool is an author of this paper.

\paragraph{Funding.} The author received no funding for this work.

\paragraph{Competing interests.} The author declares no competing interests.

\paragraph{Code availability.} All scripts and their output logs are available as ancillary files with the arXiv
version of this paper; Appendix~\ref{app:checks} lists which script verifies each computational claim.

\appendix
\section{Computational claims and scripts}\label{app:checks}
All scripts run in PARI/GP~2.17 via \texttt{cypari2}, those for $m=1$ also in \textsc{Singular}~4.3.2; symbolic means
all parameters are indeterminates. The $m=1$ scripts are in \texttt{code/m1/} (run from \texttt{code/}), with a
\texttt{README} recording each step.
\begin{longtable}{p{0.58\textwidth}p{0.34\textwidth}}
\toprule claim & script(s)\\ \midrule
\endhead
Formal corrections $\ell_2,\dots,\ell_5$ (separable vertices, exact Gaussian expansion) & \texttt{formal.gp} (\texttt{formalells})\\
$\Psi_2,\Psi_3$ symbolic; support, degrees & \texttt{psi\_sym.py}, \texttt{coprime3.py}\\
$\Psi_4$ symbolic; four edge identities & \texttt{sym4.py}\\
$\Psi_5$ top and right edges via series in $1/Y$, $1/X$ (precision $40$); corners & \texttt{top5b.py}\\
Rank-one $N_2,N_3,N_4$ symbolic, corners; $N_5$ corners certified on $16\times7$ grid & \texttt{vert.py}, \texttt{j4.py}, \texttt{cert5.py}\\
Rank-one conic identity $j=2$ (exact), $j=3$ (exact), $j=4$ (points) & \texttt{curve\_q.py}, \texttt{psi\_sym.py}, \texttt{l3curve.py}, \texttt{airy2.py}\\
Rank-one classification: $E$, $E_2$, $U(A)$, roots, $A=1$ factorisation & \texttt{rk1b.py}, \texttt{rk1c.py}, \texttt{rk1d.py}\\
Multiplicity of $(1+X)$: $M_j(-1)$ and the three lines & \texttt{mu.py}, \texttt{enum11.py} (first line of output)\\
$(1,1)$ configurations (49): exact positivity, all six conic branches, saturated Gr\"obner bases; routing to T1--T4 & \texttt{gate11.py}; branch $q_\Gamma=c\kappa^2$ alone: \texttt{gate11\_branchC.py}\\
Thin rectangles (38): all conic branches, saturated Gr\"obner bases; rational points; exact gcds at $P_1,P_2$ & \texttt{gate\_thin.py}, \texttt{gate\_thin\_pts.py}, \texttt{gate\_points\_final.py}, \texttt{thin3.json}\\
Families T1--T5 with $k$ symbolic & \texttt{fam5.py}, \texttt{t5.py}\\
Exchange symmetry (Lemma~\ref{lem:swap}), $j=2,3$ at $25$ random points & \texttt{exchange\_check.py}\\
Lemma~\ref{lem:worder}: conic weighted orders (symbolic $j=2$; exact $j\le5$ at three points) & \texttt{conic\_sym.py}, \texttt{conic\_order.py}\\
Section~\ref{sec:first}: formula \eqref{eq:R1} symbolic, bidegree, conic and $b=0$ restrictions, corners & \texttt{first\_coeff.py}, \texttt{first\_corners.py}\\
Section~\ref{sec:first}: numerical values of $r$ and $c$ at all $38$ known pairs (220 digits); Kanade--Russell values of $c$ & \texttt{first\_known.py}\\
Section~\ref{sec:count}: $\operatorname{Res}_Y(f_r,\Psi_2)=C_3^4R_{24}$ at $(s_1,\frac17)$ & \texttt{first\_res.py}\\
Section~\ref{sec:first}: the four edges of $f_r$ as rank-one numerators, symbolic & \texttt{fr\_edges.py}, \texttt{fr\_divides.py}\\
Lemma~\ref{lem:edgediv}, cases (i)--(iv); $\gcd(N_2,N_3)$ at the special pairs & \texttt{fr\_divides2.py}, \texttt{fr\_divides3.py}, \texttt{fr\_divides4.py}\\
Lemma~\ref{lem:rig}: $U$-forms of $B_2,\dots,B_5$, factorizations $(U-V)g$, $(U-V)h$, Case B ratios; identities \eqref{eq:B2}; corner fit of $f_r$ at all four corners & \texttt{check32.py}\\
Lemma~\ref{lem:ideal}: $\Psi_j\in\mathcal I^j$ at a simple and a double conic root on the boundary, $j\le4$ & \texttt{check32c.py}\\
Independent SymPy/python-flint checks of Sections~\ref{sec:first}, \ref{sec:rigidity}, \ref{sec:count}: edges of $f_r$, \eqref{eq:B2}, Lemma~\ref{lem:rig}, Lemma~\ref{lem:edgediv}, boundary ideal powers & \texttt{independent\_checks/}\newline\texttt{first\_coefficient/}\\
Known examples: Nahm points in exact number fields, $r\in\Q$ and $\gcd(f_r,\Psi_2)=1$ at all $38$ pairs & \texttt{known\_exact.py}\\
Remark on joint equations: common zeros of $f_r,\Psi_2,\Psi_3$ with $D\ne0$ & \texttt{joint\_check.py}\\
Remark on joint equations: $\Phi(r)$ at $s_1$ and in Examples~1 and~7 & \texttt{phi\_poly.py}\\
Known examples (Table~\ref{tab:known}): Nahm points, fields, vanishing of $\Psi_2,\dots,\Psi_5$ at all $38$ pairs; $\gcd(\Psi_2,\Psi_3)=1$ & \texttt{known\_examples.py}, \texttt{known\_gcd.py}\\
Euler exception (Lemma~\ref{lem:euler}): six values, $A=1$ formula; resultant sign convention & \texttt{euler\_check.py}\\
Complete $\Psi_2,\dots,\Psi_5$ at six parameter points: \texttt{formalells} versus the unshifted recurrence, coefficient by coefficient & \texttt{spot\_dump.py} (unshifted recurrence), \texttt{spot\_compare.py}\\
Independent certificates in SymPy/python-flint: top edges for all parameters, $U(A)$ at $892$ points, thin points, the branch $q_\Gamma=c\kappa^2$ & \texttt{independent\_checks/}\\
Further numerical evidence (boxes, families, coprimality tests) & \texttt{box2.py}, \texttt{fam*.py}, \texttt{coprime*.py}\\
Section~\ref{sec:m1}: effective pairs and $\rho$ on $X,Y=-1$; special pairs incl.\ $A=0,-1$ & \texttt{m1/rho\_lines.py}, \texttt{m1/special\_pairs.py}\\
Section~\ref{sec:m1}: $XY=1$ and $Y=X$ (conic condition, corners, Case B, $\Psi_4$) & \texttt{m1/xy1\_*.py}, \texttt{m1/xy\_close.py}, \texttt{m1/xy\_zagier2.py}, \texttt{m1/diag\_line*.py}\\
Section~\ref{sec:m1}: simple roots, bidegree $(2,2)$ & \texttt{m1/gen22\_*.py}\\
Section~\ref{sec:m1}: special roots (enumeration, \textsc{Singular} systems, rational points) & \texttt{m1/enum\_*.py}, \texttt{m1/sing\_*.py}, \texttt{m1/fam*.py}, \texttt{m1/special11*.py}\\
Section~\ref{sec:m1}: $b=0$; $a=1$; $a=d=1$; exchange symmetry & \texttt{m1/b0\_*.py}, \texttt{m1/a1\_*.py}, \texttt{m1/deg\_a1.py}, \texttt{m1/aa1\_*.py}, \texttt{m1/sym\_check.py}\\
Section~\ref{sec:m1}, final version: exact enumeration of all four cases; exact positivity and reducibility; Lemma~\ref{lem:tangent}; case table & \texttt{m1/fix\_*.py}, \texttt{m1/CASES.md}\\
Proposition~\ref{prop:ord22}: grid identity $V_{ij}=p_{ij}E_{ij}$; modular certificates at $P_1,\dots,P_5$ and for the families $A$, $B$; cross-check by Gr\"obner bases on all strata & \texttt{m1/ord22\_certificates.py}, \texttt{m1/fix\_ord22p.py}, \texttt{m1/fix\_points.py}\\
Section~\ref{sec:pole}: moment lemma; reciprocity (Proposition~\ref{prop:reciprocity}), symbolic for $j\le4$ & \texttt{pole/run\_recip.py} (\texttt{recip\_cert.gp}) $\to$ \texttt{pole/recip\_results.txt}; \texttt{pole/recip\_check.py}\\
Lemma~\ref{lem:norm}, rational $m$: thin points $P_1,\dots,P_4$ (certificate); diagnostics at $m=\frac32,\frac53,\frac74,\frac{25}2$ and for T1--T5 at rational $k$ & \texttt{pole/thin\_rational.py}; \texttt{pole/spot\_rational\_m.py}; \texttt{gate\_thin\_pts.py}\\
Section~\ref{sec:families}: the identities in Theorems~\ref{thm:compl}--\ref{thm:m1compl} (with $\eta,\theta,B,C$ symbolic) & \texttt{pole/families\_check.py} (needs \texttt{N1.gp}, \texttt{N2.gp} from \texttt{run\_recip.py}) \\
\bottomrule
\end{longtable}
\noindent Two implementation notes: (a) in \texttt{PsiL}/\texttt{PsiSym} the formal variable \texttt{m} of the
correction formula must be identified with the matrix parameter; (b) multivariate \texttt{factor} in PARI
can crash on large inputs, so factorisations were done univariately after elimination.
Several exclusions use $\gcd$ over sampled $x$-values of $\operatorname{Res}_Y(\Psi_2,\Psi_3)$; this is a multiple of the
true exceptional polynomial, so an empty root set is a proof, not evidence.

\noindent Dependence and validation. Everything rests on \cite[Thm.~2.1]{Miz} (the radial expansion and the formal
Gaussian description of its coefficients) and on the argument of Section~\ref{sec:setting} showing that
$\ell_j=0$ is necessary. Two independent implementations of the formal integral agree on complete polynomials
(Appendix~\ref{app:checks}). As further validation, the expansion gives exact vanishing of $\ell_2,\ell_3,\ell_4$
and $\ell_1=c$ for nine known modular sums, and agrees with brute-force summation.


\addcontentsline{toc}{section}{References}
\begin{thebibliography}{99}
\bibitem{Bre} D.~M.~Bressoud, \emph{A generalization of the Rogers--Ramanujan identities for all moduli}, J. Combin. Theory Ser. A \textbf{27} (1979), 64--68.
\bibitem{CGZ} F.~Calegari, S.~Garoufalidis and D.~Zagier, \emph{Bloch groups, algebraic K-theory, units, and Nahm's conjecture}, Ann. Sci. \'Ec. Norm. Sup\'er. (4) \textbf{56} (2023), 383--426.
\bibitem{CW} Z.~Cao and L.~Wang, \emph{Multi-sum Rogers--Ramanujan type identities}, J. Math. Anal. Appl. \textbf{522} (2023), 126960.
\bibitem{Cap} S.~Capparelli, \emph{On some representations of twisted affine Lie algebras and combinatorial identities}, J. Algebra \textbf{154} (1993), 335--355.
\bibitem{GZ} S.~Garoufalidis and D.~Zagier, \emph{Asymptotics of Nahm sums at roots of unity}, Ramanujan J. \textbf{55} (2021), 219--238.
\bibitem{Ink} K.~Inkeri, \emph{The real roots of Bernoulli polynomials}, Ann. Univ. Turku. Ser. A I \textbf{37} (1959), 1--20.
\bibitem{KR} S.~Kanade and M.~C.~Russell, \emph{IdentityFinder and some new identities of Rogers--Ramanujan type}, Exp. Math. \textbf{24} (2015), 419--423.
\bibitem{Kan} S.~Kanade, \emph{Searching for modular companions}, arXiv:1902.05199 (2019).
\bibitem{Kur} K.~Kur\c{s}ung\"oz, \emph{Andrews--Gordon type series for Capparelli's and G\"ollnitz--Gordon identities}, J. Combin. Theory Ser. A \textbf{165} (2019), 117--138.
\bibitem{Kur2} K.~Kur\c{s}ung\"oz, \emph{Andrews--Gordon type series for Kanade--Russell conjectures}, Ann. Comb. \textbf{23} (2019), 835--888.
\bibitem{LW} Z.~Li and L.~Wang, \emph{Rogers--Ramanujan type identities involving double, triple and quadruple sums}, arXiv:2306.17085.
\bibitem{Miz} Y.~Mizuno, \emph{Remarks on Nahm sums for symmetrizable matrices}, Ramanujan J. \textbf{66} (2025), article 62; arXiv:2305.02267.
\bibitem{Miz26} Y.~Mizuno, \emph{The three Kanade--Russell identities modulo nine}, arXiv:2609.09653 (2026).
\bibitem{Nah} W.~Nahm, \emph{Conformal field theory and torsion elements of the Bloch group}, in: Frontiers in Number Theory, Physics, and Geometry II, Springer, 2007, 67--132.
\bibitem{Mat} H.~Matsumura, \emph{Commutative Ring Theory}, Cambridge Stud. Adv. Math. \textbf{8}, Cambridge University Press, 1986.
\bibitem{Sing} W.~Decker, G.-M.~Greuel, G.~Pfister and H.~Sch\"onemann, \textsc{Singular} 4-3-2 --- A computer algebra system for polynomial computations, 2022, \url{https://www.singular.uni-kl.de}.
\bibitem{Ost} A.~Ostrowski, \emph{\"Uber die Bedeutung der Theorie der konvexen Polyeder f\"ur die formale Algebra}, Jahresber. Deutsch. Math.-Verein. \textbf{30} (1921), 98--99.
\bibitem{VZ} M.~Vlasenko and S.~Zwegers, \emph{Nahm's conjecture: asymptotic computations and counterexamples}, Commun. Number Theory Phys. \textbf{5} (2011), 617--642.
\bibitem{WW} B.~Wang and L.~Wang, \emph{Proofs of Mizuno's conjectures on generalized rank two Nahm sums}, Trans. Amer. Math. Soc. \textbf{379} (2026), 459--486; arXiv:2308.14728.
\bibitem{Xia} E.~X.~W.~Xia, \emph{The modulo 9 Kanade--Russell identities and their Nahm-sum duals}, arXiv:2609.08816 (2026).
\bibitem{Z} D.~Zagier, \emph{The dilogarithm function}, in: Frontiers in Number Theory, Physics, and Geometry II, Springer, 2007, 3--65.
\end{thebibliography}
\end{document}